\documentclass[12 pt]{article}
	\usepackage[english]{babel}
	\usepackage[utf8]{inputenc}  
	\usepackage[T1]{fontenc}
	\usepackage{amssymb}
	\usepackage[mathscr]{euscript}
	\usepackage{stmaryrd}
	\usepackage{amsmath}
	\usepackage{tikz}
	\usepackage{caption}
	\usepackage[all,cmtip]{xy}
	\usepackage{amsthm}
	\usepackage{varioref}
	\usepackage{geometry}
	\usepackage{lmodern}
	\usepackage[unicode]{hyperref}
	 \usepackage{fancyhdr}
	\usepackage{xargs}
	\usepackage{ifthen}
	\usepackage{caption}
	\usepackage{float}
	\theoremstyle{plain}
	\newcounter{n}
	\numberwithin{n}{section}
	\newtheorem{theo}[n]{Theorem}
	\labelformat{theo}{Theorem~#1}
	\labelformat{rmq}{remarque~#1}
	\newtheorem{cor}[n]{Corollary}
	\newtheorem{lm}[n]{Lemma}
	\labelformat{lm}{lemma~#1}
	
	\newtheorem{prop}[n]{Proposition}
	\labelformat{prop}{proposition~#1}
	{\theoremstyle{definition}

	\newtheorem{nt}[n]{Notation}
    	
	\newtheorem{df}[n]{Definition}} 
	\newtheorem*{thmnonnum}{Theorem}
	
 \usetikzlibrary{patterns}
	\newcommand\lettre{{\mathcal G}}
	\renewcommand\epsilon{\varepsilon}
	\renewcommand\phi{\varphi}
	\newcommand\R{\mathbb{R}}
	\newcommand\s{\mathbb{S}}

	\newcommand\graphesJn{\widetilde{\mathcal D_k}}
	\newcommand\diagrammes{{\mathcal A_k}}
	\newcommand\diagrammesc{{\check{\mathcal A_k}}}
	\newcommand\nonprimitifs{{\mathcal N_k}}
	\newcommand\diagrammess{{\mathcal T_k}}
	\newcommand\primitifs{\mathcal P_k}
	\newcommandx\ordres[1][1=\Gamma]{{\mathfrak{O}(#1)}}
\newcommand\und{\underline}
\newcommand\Card{{\mathrm{Card}}}

\newcommand\sgn{\mathrm{sgn}}

\newcommand{\punct}[1]{{#1^\circ}}
	
\newcommand\sphamb{{M}}

\newcommand\spamb{{\punct \sphamb}}
\newcommandx\bouleambiante[1][1= M]{B(#1)}
\newcommandx\bM[1][1= M]{B(#1)}

\newcommandx\nbdinf[1][1= ]{\ifthenelse{\equal{#1}{}}{\punct{B_\infty}}{\punct{B_\infty(#1)}}}

\newcommandx\sommets[2][1=\Gamma,2= \empty]{\ifthenelse{\equal{#2}{}} {V(#1)}{V^*(#1)} }

\newcommand\sommetsinternes[1][\Gamma]{V_i(#1)}

\newcommand\sommetsexternes[1][\Gamma]{V_e(#1)}

\newcommand\aretes[1][\Gamma]{E(#1)}

\newcommand\aretesinternes[1][\Gamma]{E_i(#1)}

\newcommand\aretesexternes[1][\Gamma]{E_e(#1)}

\newcommand\Prop{{\omega}}

\newcommand\graphesnum{\widetilde{\mathcal G_k}}

\newcommand\univalents[1][\Gamma]{V_i(#1)}

\newcommandx\confignoeud[3][1=\Gamma, 2=\psi,3=\empty ]{C_{#1}#3(#2)}

\newcommand\config{C_2}

\newcommand\configM{\config(\spamb)}

\renewcommand\d{\mathrm d}

\newcommandx\formearete[2][1=e,2=\sigma]{{\omega^F_{#1,#2}}}

\newcommand{\crc}{circle (0.1)}
\newcommand{\edgi}{\draw[->, >=latex]}
\newcommand{\edge}{\draw[->, >=latex, dashed]}

\newcommand{\figJAC}{
\begin{tikzpicture}
\begin{scope}[scale=.5]
\fill(0,0) circle (0.1) (0, 2) circle (0.1) (0,-2) circle (0.1);
\fill (-1,0) circle (0.1) (-1, 2) circle (0.1) (-1,-2) circle (0.1);
\fill (-2,0) circle (0.1) (-2, 2) circle (0.1) (-2,-2) circle (0.1);
\fill (-4, 0) circle (0.1); 
\draw[dashed]  (0,2) -- (-.9,2);
\draw[dashed]  (-1.1, 2) -- (-1.9, 2);
\draw[dashed] (-2.06, 1.94) -- ++(225: 2.72) ;
\draw[dashed]  (-0.1,-2) -- (-.9,-2);
\draw[dashed]  (-1, -2) -- (-1.9, -2);
\draw[dashed]  (-2, -2) -- ++(135: 2.72) ;
\draw[dashed]  (-0.1,0) -- (-.9,0);
\draw[dashed] (-2.1, 0) --  (-3.9,0) ;
\draw[dashed] (-2, 0.1) --  (-2,1.9) ;
\draw[dashed] (-2, -0.1) --  (-2,-1.9) ;
\draw[dashed] (-1, 0.1) --  (-1,1.9) ;
\draw[dashed] (-1, -0.1) --  (-1,-1.9) ;
\draw[thick, ->] (0, -3) -- (0, 3);
\end{scope}
\begin{scope}[scale=.5]
\fill (2, 0) circle (0.1);
\fill (2, 0) ++ (0: 1) circle (0.1);
\fill (2, 0) ++ (120: 1) circle (0.1);
\fill (2, 0) ++ (240: 1) circle (0.1);
\draw[dashed] (2, 0) ++ (0:.1) -- ++(0:.8);
\draw[dashed] (2, 0) ++ (120:.1) -- ++(120:.8);
\draw[dashed] (2, 0) ++ (240:.1) -- ++(240:.8);
\draw[dashed] (2, 0) ++ (120:1) -- ++(-30:1.67);
\draw[dashed] (2, 0) ++ (240:1) -- ++(30:1.67);
\draw[dashed] (2, 0) ++ (120:1) -- ++(-90:1.67);
\end{scope}
\end{tikzpicture}
}

\newcommand\figureas{
\begin{tikzpicture}[scale=1.25] \useasboundingbox (-.4,-.3) rectangle (.4,.4);
\draw [dashed] (0,0) -- (0,-.4) (0,0) -- (50:.4) (0,0) -- (130:.4);
\fill (0,0) circle (1.5pt);
\end{tikzpicture} = - \begin{tikzpicture}[scale=1.25] \useasboundingbox (-.4,-.4) rectangle (.4,.3);
\draw [dashed](0,0) -- (0,-.4);
\draw [dashed] (0,0) .. controls (50:.3) ..  (130:.5);
\draw [draw=white,dashed, double=black,very thick] (0,0) .. controls (130:.3) ..  (50:.5);
\fill (0,0) circle (1.5pt);
\end{tikzpicture}}

\newcommand{\stuy}{\begin{tikzpicture}[scale=1.25, rotate=90] \useasboundingbox (-.1,-.75) rectangle (.7,.4);
\draw [->,thick] (0,-.5) -- (.6,-.5);
\draw [dashed] (.3,-.5) -- (.3,0) -- (.075,.3) (.3,0) -- (.525,.3);
\fill (.3,-.5) circle (1.5pt) (.3,0) circle (1.5pt);
\draw (.3, -.65) node {\tiny $\Gamma$};
\end{tikzpicture}}
\newcommand{\stui}{\begin{tikzpicture}[scale=1.25, rotate=90] \useasboundingbox (-.1,-.75) rectangle (.8,.4);
\draw [->,thick] (0,-.2) -- (.7,-.2);
\draw [dashed] (.15,-.2) -- (.15,.2)  (.45,.2) -- (.45,-.2);
\fill (.15,-.2) circle (1.5pt) (.45,-.2) circle (1.5pt);
\draw (.3, -.35) node {\tiny $\Gamma_1$};
\end{tikzpicture}}
\newcommand{\stux}{\begin{tikzpicture}[scale=1.25, rotate=90] \useasboundingbox (-.1,-.75) rectangle (.8,.4);
\draw [->,thick] (0,-.2) -- (.7,-.2);
\draw [dashed] (.15,.2) -- (.45,-.2);
\draw [draw=white,double=black,very thick, dashed] (.15,-.2) -- (.45,.2);
\fill (.15,-.2) circle (1.5pt) (.45,-.2) circle (1.5pt);
\draw (.3, -.35) node {\tiny $\Gamma_2$};
\end{tikzpicture}}

\newcommand\figurestu{
\stuy =\stui - \stux
}

\newcommand{\ihxone}{\begin{tikzpicture} \useasboundingbox (-.4,-.7) rectangle (.4,.6);
\draw [dashed] (-90:.6) -- (0,0) -- (30:.6) (0,0) -- (150:.6);
\draw [dashed] (-.5,-.25) -- (0,-.25);
\fill (0,-.25) circle (1.5pt) (0,0) circle (1.5pt);
\draw (-.2, -.6) node {\tiny $\Gamma_1$};
\end{tikzpicture}}
\newcommand{\ihxtwo}{\begin{tikzpicture} \useasboundingbox (-.4,-.7) rectangle (.4,.6);
\draw [dashed] (-90:.6) -- (0,0) -- (30:.6) (0,0) -- (150:.6);
\draw [dashed,out=0,in=-60, draw=white,double=black,very thick] (-.5,-.25) to (30:.25);
\fill (30:.25) circle (1.5pt) (0,0) circle (1.5pt);
\draw (-.2, -.6) node {\tiny $\Gamma_2$};
\end{tikzpicture}}
\newcommand{\ihxthree}{\begin{tikzpicture} \useasboundingbox (-.4,-.7) rectangle (.4,.6);
\draw [dashed]  (-90:.6) -- (0,0) -- (30:.6) (0,0) -- (150:.6);
\draw [dashed, draw=white,double=black,very thick] (-.5,-.25) .. controls (-.2,-.25) and  (.25,-.2)  .. (.25,0)  .. controls (.25,.2) and (120:.35) .. (150:.25);
\fill (150:.25) circle (1.5pt) (0,0) circle (1.5pt);
\draw (-.2, -.6) node {\tiny $\Gamma_3$};
\end{tikzpicture}}

\newcommand\figureihx{
\ihxone+\ihxtwo+\ihxthree= 0
}

\newcommand\graphek{\begin{tikzpicture}[scale = 0.4, xscale=-1]
\useasboundingbox (0, -3) rectangle (3, 3);
\draw [->,thick] (0, -3.2) -- (0,3.2); 
\draw [dashed](0, -2.5) -- (1, -2.5); 
\draw [dashed](0, -1.5) -- (1, -1.5); 
\draw [dashed](0, 1.5) -- (1, 1.5); 
\draw [dashed](0, 2.5) -- (1, 2.5); 
\draw [dashed](1, -2.5) -- (1, -1.5); 
\draw [dashed](1, -1.5) -- (1, 1.5); 
\draw [dashed](1, 1.5) -- (1, 2.5); 
\draw [dashed](1, 2.5) to [bend left = 80] (1, -2.5); 
\fill (0, -2.5) circle (.1) (1, -2.5) circle (.1);
\fill (0, -1.5) circle (.1) (1, -1.5) circle (.1);
\fill (0, 1.5) circle (.1) (1, 1.5) circle (.1);
\fill (0, 2.5) circle (.1) (1, 2.5) circle (.1);
\draw (1.4, -3.4) node {$\Gamma_k$};
\end{tikzpicture}
}

\newcommand\graphekA{\begin{tikzpicture}[scale=.8]
\draw [dashed, <-, >= latex] (3.1,0) -- (3.9,0) ;
\draw [dashed, ->, >= latex] (1,-0.1) to[bend right=60] (4,-0.1) ;
\draw [dotted] (1.2,0) -- (1.7,0) ;
\draw [dashed, <-, >= latex] (2.1,0) -- (2.9,0) ;
\draw [dashed, <-, >=latex] (1.1,0) -- (1.2,0) ;
\draw [dashed] (1.7,0) -- (1.9,0) ;
\fill (1, 1) circle (0.1);
\fill (2, 1) circle (0.1);
\fill (3, 1) circle (0.1);
\fill (4, 1) circle (0.1);
\draw [dashed, ->, >= latex] (1,0.9) -- (1,0.1) ;
\draw [dashed, ->, >= latex] (2,0.9) -- (2,0.1) ;
\draw [dashed, ->, >= latex] (3,0.9) -- (3,0.1) ;
\draw [dashed, ->, >= latex] (4,0.9) -- (4,0.1) ;
\draw (1, 0) circle (0.1);
\draw (2, 0) circle (0.1);
\draw (3, 0) circle (0.1);
\draw (4, 0) circle (0.1);
\draw (1,1) circle (0.1) ++(0,0.3) node {$v_{k}$};
\draw (2,1) circle (0.1) ++(0,0.3) node {$v_{3}$};
\draw (3,1) circle (0.1) ++(0,0.3) node {$v_{2}$};
\draw (4,1) circle (0.1) ++(0,0.3) node {$v_{1}$};
\draw (2.5, 2) node {$(\Gamma_k^{BCR}, \rho_a)$};
\end{tikzpicture}}

\newcommand\graphekB{\begin{tikzpicture}[scale=.8]
\draw [dashed, <-, >= latex] (1.1,0) -- (1.9,0) ;
\draw [dashed, ->, >= latex] (1,-0.1) to[bend right=60] (4, -0.1) ;
\draw [dotted] (3.2,0) -- (3.7,0) ;
\draw [dashed, <-, >= latex] (2.1,0) -- (2.9,0) ;
\draw [dashed, <-, >=latex] (3.1,0) -- (3.2,0) ;
\draw [dashed] (3.7,0) -- (3.9,0) ;
\draw (1, 0) circle (0.1);
\draw (2, 0) circle (0.1);
\draw (3, 0) circle (0.1);
\draw (4, 0) circle (0.1);
\fill (1, 1) circle (0.1);
\fill (2, 1) circle (0.1);
\fill (3, 1) circle (0.1);
\fill (4, 1) circle (0.1);
\draw (1,1) circle (0.1) ++(0,0.3) node {$v_{1}$};
\draw (2,1) circle (0.1) ++(0,0.3) node {$v_{ 2}$};
\draw (3,1) circle (0.1) ++(0,0.3) node {$v_{ 3}$};
\draw (4,1) circle (0.1) ++(0,0.3) node {$v_{k}$};
\draw [dashed, ->, >= latex] (1,0.9) -- (1,0.1) ;
\draw [dashed, ->, >= latex] (2,0.9) -- (2,0.1) ;
\draw [dashed, ->, >= latex] (3,0.9) -- (3,0.1) ;
\draw [dashed, ->, >= latex] (4,0.9) -- (4,0.1) ;
\draw (2.5, 2) node {$(\Gamma_k^{BCR}, \rho_b)$};
\end{tikzpicture}}

\newcommand\gammaun{
\begin{tikzpicture}[scale = .6, xscale =-1]
\draw [thick, ->] (0, -2) -- (0, 2);
\draw [dashed] (0, -1) -- (1, -1.5) (0, 1) -- (1, 1.5);
\fill (0, -1) circle (0.1) (0,1) circle (0.1); 
\draw (-.5, -1) node {$v$} (-.5, 1) node {$w$};
\draw (.85, -1) node {$e$} (.85, 1) node {$f$};
\draw (0, -2.5) node {$\Gamma_J^{(1)}$};
\end{tikzpicture}}

\newcommand\gammadeux{
\begin{tikzpicture}[scale = .6, xscale=-1]
\draw [thick, ->] (0, -2) -- (0, 2);
\draw [dashed](0, -1) -- (1, 1.5) (0, 1) -- (1, -1.5);
\fill (0, -1) circle (0.1) (0,1) circle (0.1); 
\draw (-.5, -1) node {$w$} (-.5, 1) node {$v$};
\draw (1.2, -1) node {$e$} (1.2, 1) node {$f$};
\draw (0, -2.5) node {$\Gamma_J^{(2)}$};
\end{tikzpicture}}

\newcommand\gammatrois{
\begin{tikzpicture}[scale = .6, xscale=-1]
\draw [thick, ->] (0, -2) -- (0, 2);
\draw [dashed](0, 0) -- (1, 0);
\draw [dashed] (1,0) -- (1.5, 1.5) (1, 0) -- (1.5, -1.5);
\fill (0, 0) circle (0.1) (1,0) circle (0.1); 
\draw (-.6, 0) node {$u$} (1.25, 0) node {$t$};
\draw (1.7, -1) node {$e$} (1.7, 1) node {$f$};
\draw (.5, .35) node {$h$};
\draw (0, -2.5) node {$\Gamma_J$};
\end{tikzpicture}}

\newcommand\gammaBBa{
\begin{tikzpicture}[scale = .9]
\useasboundingbox (-1, -1) rectangle (3, 1.5);
\fill (0, 0) circle (0.1) (1, 0) circle (0.1) ;
\draw [->, >= latex, dashed] (-1, 0) -- (-.1, 0);
\draw [->, >= latex] (.1, 0) -- (.9, 0);
\draw [->, >= latex, dashed] (1.1, 0) -- (1.9, 0);
\draw (.5, -.6) node {$\Gamma$};
\draw (0, -.35) node {$v$};
\draw (1, -.35) node {$w$};
\draw (-.5, .25) node {$e$};
\draw (1.5, .25) node {$f$};
\end{tikzpicture}
}

\newcommand\gammaBBastar{
\begin{tikzpicture}[scale = .9]
\useasboundingbox (-1, -1) rectangle (3, 1.5);
\draw (0, 0) circle (0.1);
\fill  (0, -1) circle (0.1) ;
\draw [->, >= latex, dashed] (-1, 0) -- (-.1, 0);
\draw [->, >= latex, dashed] (.1, 0) -- (.9, 0);
\draw [->, >= latex, dashed] (0, -.9) -- (0, -0.1);
\draw (.75, -.6) node {$\Gamma^*$};
\draw (0, .35) node {$t$};
\draw (-.5, .25) node {$e$};
\draw (.5, .25) node {$f$};
\draw (0.15, -.5) node {$h$};
\draw (-.25, -1) node {$u$};
\end{tikzpicture}
}

\newcommand\gammaBBb{
\begin{tikzpicture}[scale = .9]
\useasboundingbox (-1, -1) rectangle (3, 1.5);
\fill (0, 0) circle (0.1) (1, 0) circle (0.1) ;
\draw [<-, >= latex, dashed] (-1, 0) -- (-.1, 0);
\draw [<-, >= latex] (.1, 0) -- (.9, 0);
\draw [<-, >= latex, dashed] (1.1, 0) -- (1.9, 0);
\draw (.5, -.6) node {$\Gamma$};
\draw (0, -.35) node {$v$};
\draw (1, -.35) node {$w$};
\draw (-.5, .25) node {$e$};
\draw (1.5, .25) node {$f$};
\end{tikzpicture}
}

\newcommand\gammaBBbstar{
\begin{tikzpicture}[scale = .9]
\useasboundingbox (-1, -1) rectangle (3, 1.5);
\draw (0, 0) circle (0.1);
\fill  (0, 1) circle (0.1) ;
\draw [<-, >= latex, dashed] (-1, 0) -- (-.1, 0);
\draw [<-, >= latex, dashed] (.1, 0) -- (.9, 0);
\draw [->, >= latex, dashed] (0, .9) -- (0, 0.1);
\draw (.75, -.6) node {$\Gamma^*$};
\draw (0, -.35) node {$t$};
\draw (-.5, .25) node {$e$};
\draw (.5, .25) node {$f$};
\draw (0.15,.5) node {$h$};
\draw (-.25, 1) node {$u$};
\end{tikzpicture}
}

\newcommand\gammaTBa{
\begin{tikzpicture}[yscale=-1, scale = .9]
\useasboundingbox (-1, -1.5) rectangle (3, 1.5);
\fill (0, 0) circle (0.1) (1, 0) circle (0.1) (0,1) circle (0.1) ;
\draw [->, >= latex] (-1, 0) -- (-.1, 0);
\draw [->, >= latex] (.1, 0) -- (.9, 0);
\draw [->, >= latex, dashed] (1.1, 0) -- (1.9, 0);
\draw [->, >=latex, dashed] (0,.9) --(0,.1);
\draw (.5, -.75) node {$\Gamma$};
\draw (0, -.35) node {$v$};
\draw (1, -.35) node {$w$};
\draw (.15, .5) node {$e$};
\draw (1.5, .25) node {$f$};
\draw (0.3, 1) node {$x$};
\end{tikzpicture}
}
\newcommand\gammaTBastar{
\begin{tikzpicture}[yscale = -1,scale = .9]
\useasboundingbox (-1, -1.5) rectangle (3, 1.5);
\fill (0, 0) circle (0.1) (0,1) circle (0.1) ;
\draw (1, 0) circle (0.1);
\draw [->, >= latex] (-1, 0) -- (-.1, 0);
\draw [->,dashed , >= latex] (.1, 0) -- (.9, 0);
\draw [->, >= latex, dashed] (1.1, 0) -- (1.9, 0);
\draw [<-, >=latex, dashed] (1,0) ++(135:.1) -- (0, 1);
\draw (.5, -.75) node {$\Gamma^*$};
\draw (0, -.35) node {$u$};
\draw (1, -.35) node {$t$};
\draw (.8, .6) node {$e$};
\draw (1.5, .25) node {$f$};
\draw (.5, -.25) node {$h$};
\draw (0.3, 1) node {$x$};
\end{tikzpicture}
}

\newcommand\gammaTBb{
\begin{tikzpicture}[scale = .9]
\useasboundingbox (-1, -1) rectangle (3, 1.5);
\fill (0, 0) circle (0.1) (1, 0) circle (0.1) (0,1) circle (0.1) ;
\draw [<-, >= latex] (-1, 0) -- (-.1, 0);
\draw [<-, >= latex] (.1, 0) -- (.9, 0);
\draw [<-, >= latex, dashed] (1.1, 0) -- (1.9, 0);
\draw [->, >=latex, dashed] (0,.9) --(0,.1);
\draw (.5, -.75) node {$\Gamma$};
\draw (0, -.35) node {$v$};
\draw (1, -.35) node {$w$};
\draw (.15, .5) node {$e$};
\draw (1.5, .25) node {$f$};
\draw (0.3, 1) node {$x$};
\end{tikzpicture}
}
\newcommand\gammaTBbstar{
\begin{tikzpicture}[scale = .9]
\useasboundingbox (-1, -1) rectangle (3, 1.5);
\fill (0, 0) circle (0.1) (0,1) circle (0.1) ;
\draw (1, 0) circle (0.1);
\draw [<-, >= latex] (-1, 0) -- (-.1, 0);
\draw [<-,dashed , >= latex] (.1, 0) -- (.9, 0);
\draw [<-, >= latex, dashed] (1.1, 0) -- (1.9, 0);
\draw [<-, >=latex, dashed] (1,0) ++(135:.1) -- (0, 1);
\draw (.5, -.75) node {$\Gamma^*$};
\draw (0, -.35) node {$u$};
\draw (1, -.35) node {$t$};
\draw (.8, .6) node {$e$};
\draw (1.5, .25) node {$f$};
\draw (.5, -.25) node {$h$};
\draw (0.3, 1) node {$x$};
\end{tikzpicture}
}

\newcommand\gammaBTa{
\begin{tikzpicture}[yscale=-1,scale = .9]
\useasboundingbox (-1, -1.5) rectangle (3, 1);
\fill (0, 0) circle (0.1) (1, 0) circle (0.1) (1,1) circle (0.1) ;
\draw [->,dashed, >= latex] (-1, 0) -- (-.1, 0);
\draw [->, >= latex] (.1, 0) -- (.9, 0);
\draw [->, >= latex] (1.1, 0) -- (1.9, 0);
\draw [->, >=latex, dashed] (1,.9) --(1,.1);
\draw (.5, -.75) node {$\Gamma$};
\draw (0, -.35) node {$v$};
\draw (1, -.35) node {$w$};
\draw (-.5, .25) node {$e$};
\draw (1.25, .5) node {$f$};
\draw (1.3, 1) node {$x$};
\end{tikzpicture}
}
\newcommand\gammaBTastar{
\begin{tikzpicture}[yscale = -1, scale = .9]
\useasboundingbox (-1, -1.5) rectangle (3, 1);
\fill (1, 0) circle (0.1) (1,1) circle (0.1) ;
\draw (0, 0) circle (0.1);
\draw [->, dashed, >= latex] (-1, 0) -- (-.1, 0);
\draw [->,dashed , >= latex] (.1, 0) -- (.9, 0);
\draw [->, >= latex] (1.1, 0) -- (1.9, 0);
\draw [<-, >=latex, dashed] (0,0) ++(45:.1) -- (1, 1);
\draw (.5, -.75) node {$\Gamma^*$};
\draw (0, -.35) node {$t$};
\draw (1, -.35) node {$u$};
\draw (-.5, .25) node {$e$};
\draw (.2, .65) node {$f$};
\draw (.5, -.25) node {$h$};
\draw (1.3, 1) node {$x$};
\end{tikzpicture}
}

\newcommand\gammaBTb{
\begin{tikzpicture}[scale = .9]
\useasboundingbox (-1, -1) rectangle (3, 1.5);
\fill (0, 0) circle (0.1) (1, 0) circle (0.1) (1,1) circle (0.1) ;
\draw [<-,dashed, >= latex] (-1, 0) -- (-.1, 0);
\draw [<-, >= latex] (.1, 0) -- (.9, 0);
\draw [<-, >= latex] (1.1, 0) -- (1.9, 0);
\draw [->, >=latex, dashed] (1,.9) --(1,.1);
\draw (.5,- .75) node {$\Gamma$};
\draw (0, -.35) node {$v$};
\draw (1, -.35) node {$w$};
\draw (-.5, .25) node {$e$};
\draw (1.25, .5) node {$f$};
\draw (1.3, 1) node {$x$};
\end{tikzpicture}
}
\newcommand\gammaBTbstar{
\begin{tikzpicture}[scale = .9]
\useasboundingbox (-1, -1) rectangle (3, 1.5);
\fill (1, 0) circle (0.1) (1,1) circle (0.1) ;
\draw (0, 0) circle (0.1);
\draw [<-, dashed, >= latex] (-1, 0) -- (-.1, 0);
\draw [<-,dashed , >= latex] (.1, 0) -- (.9, 0);
\draw [<-, >= latex] (1.1, 0) -- (1.9, 0);
\draw [<-, >=latex, dashed] (0,0) ++(45:.1) -- (1, 1);
\draw (.5,- .75) node {$\Gamma^*$};
\draw (0, -.35) node {$t$};
\draw (1, -.35) node {$u$};
\draw (-.5, .25) node {$e$};
\draw (.2, .65) node {$f$};
\draw (.5, -.25) node {$h$};
\draw (1.3, 1) node {$x$};
\end{tikzpicture}
}

\newcommand\wczero{
\begin{tikzpicture}
\draw[thick, ->] (0, - 1.5) -- (0, -1.3) (0, -1.1) -- (0, -.3) (0, .3) -- (0,1.1) (0,1.3)--(0, 1.5);
\fill (0, -.75) circle (0.1) (0, .75) circle (0.1);
\draw[dashed] (0, -.75) to[bend left = 70] (0, .75) ;
\draw[thick, ->] (4.5, - 1.5) -- (4.5, -1.3) (4.5, -1.1) -- (4.5, -.9) (4.5, -.6) -- (4.5, -.3) 
(4.5, .3) -- (4.5, .6) 
(4.5, .9)--(4.5, 1.1) (4.5, 1.3) --(4.5, 1.5);
\draw[thick] (4.5, -.9) to[bend left = 70] (4.5, .9) ;
\draw[thick] (4.5, -.6) to[bend left = 70] (4.5, .6) ;
\draw[->] (.5, 0) -- (3.8, 0);
\end{tikzpicture}
}

\newcommand\wcun{
\begin{tikzpicture}[xscale=.75]
\draw[thick, ->] (0, - 1.5) -- (0, 1.5);
\fill (0, -.75) circle (0.1) (0, .75) circle (0.1);
\draw[dashed] (0, -.75) to[bend left = 70] (0, .75) ;
\draw[thick, ->] (4.5, - 1.5) -- (4.5, -.95) (4.5, -.65) -- (4.5, .65) (4.5, .95)--(4.5, 1.5);
\draw[thick] (4.5, -.95) to[bend left = 70] (4.5, .95) ;
\draw[thick] (4.5, -.65) to[bend left = 70] (4.5, .65) ;
\draw[->] (.5, 0) -- (3.8, 0);
\draw (4.5, -1.75) node{\scriptsize{ $c=1$, $w_C(\Gamma_a)=0$}};
\draw (.25, -1.75) node{\scriptsize{ $\Gamma_a$}};
\end{tikzpicture}
}

\newcommand\wcdeux{
\begin{tikzpicture}[xscale=.75]
\draw[thick, ->] (0, - 1.5) -- (0, 1.5);
\fill (0, -.9) circle (0.1) (0, -.3) circle (0.1) (0, .3) circle (0.1)   (0, .9) circle (0.1);
\draw[dashed] (0, -.9) to[bend left = 70] (0, .3) ;
\draw[dashed] (0, .9) to[bend left = 70] (0, -.3) ;
\draw[thick, ->] (4.5, - 1.5) -- (4.5, -1.05) (4.5, -.75) -- (4.5, -.45) (4.5, -.15)--(4.5, .15)
(4.5, .45) -- (4.5, .75) (4.5, 1.05) -- (4.5, 1.5);
\draw[thick] (4.5, -1.05) to[bend left = 70] (4.5, .45) ;
\draw[thick] (4.5, -.75) to[bend left = 70] (4.5, .15) ;
\draw[thick] (4.5, 1.05) to[bend left = 70] (4.5, -.45) ;
\draw[thick] (4.5, .75) to[bend left = 70] (4.5,-.15) ;

\draw[->] (.5, 0) -- (3.8, 0);
\draw (4.5, -1.75) node{\scriptsize{ $c=0$, $w_C(\Gamma_b)=1$}};
\draw (.25, -1.75) node{\scriptsize{ $\Gamma_b$}};
\end{tikzpicture}
}

\newcommand\guna{
\begin{tikzpicture}
\fill (0,0) circle (0.1) (1, 0) circle (0.1) (0,1) circle (0.1) (1,1) circle (0.1);
\draw[->, >= latex] (2, 0) -- (1.1, 0);
\draw[->, >= latex] (.9, 0) -- (.1, 0);
\draw[->, >= latex] (-.1, 0) -- (-.9, 0);
\draw[->, dashed, >= latex] (0, .9) -- (0, .1);
\draw[->,dashed,  >= latex] (1, .9) -- (1, .1);
\draw (0, -.25) node{$w$} (1, -.25) node{$v$};
\draw (0, 1.25) node{$x$} (1, 1.25) node{$y$};
\draw (-.25, .5) node{$f$} (1.25, .5) node {$e$};
\end{tikzpicture}
}

\newcommand\gnc{
\begin{tikzpicture}
\draw (3, 0) circle (0.1);
\fill (3, 1) circle (0.1) (4, 0) circle (0.1) ++(-45:1) circle (0.1);
\draw[<-, >=latex] (4,0)++(-45:.1) -- ++(-45: .8);
\draw[->, dashed, >= latex] (3.9,0)--(3.1,0);
\draw[->, dashed, >= latex] (2.9,0)--(2.4,0);
\draw[->, dashed, >= latex] (3,.9)--(3, .1);
\draw (5, -.5) node {$v$}
(4.2, .2) node {$w$}
(3.2, .2) node {$t$}
(2.7, 1.1) node{$x$};
\end{tikzpicture}
}

\newcommand\gnccyc{
\begin{tikzpicture}
\fill (0,0) circle (.1) (1, 0) circle (.1) ++(-45:1) circle(.1);
\draw[<-, >=latex] (1,0) ++(-45:.1) --++(-45:.8);
\draw[<-, >=latex, dashed] (.1, 0) -- (.9,0);
\draw[->, >=latex] (-135:.1) --(-135:.8);
\draw (-.2, .3) node {$w_2$} (1.2, .3) node {$w_1$}
(1.7, -1.1) node {$v$};
\end{tikzpicture}
}

\newcommand\gncleg{
\begin{tikzpicture}
\fill (0,0) circle (.1) (1, 0) circle (.1) ++(-45:1) circle(.1);
\draw[->, >=latex] (1,0) ++(-45:.9) --(0.05, -0.05);
\draw[<-, >=latex, dashed] (.1, 0) -- (.9,0);
\draw[->, >=latex] (-135:.1) --(-135:.8);
\draw (-.2, .3) node {$w_2$} (1.2, .3) node {$w_1$}
(1.7, -1.1) node {$v$};
\end{tikzpicture}
}

\title{Bott-Cattaneo-Rossi invariants for long knots in asymptotic homology $\R^3$}
	\author{David Leturcq\footnote{Institut Fourier, Université-Grenoble-Alpes}}
	\date{ }
\begin{document}
\maketitle
\begin{abstract}
In this article, we express the Alexander polynomial of null-homologous long knots in 
punctured rational homology $3$-spheres 
in terms of integrals over configuration spaces.
To get such an expression,
we use a previously established formula,
which gives generalized Bott-Cattaneo-Rossi invariants in terms
of the Alexander polynomial and vice versa,
and we relate these Bott-Cattaneo-Rossi invariants to the perturbative
expansion of Chern-Simons theory.


\end{abstract}
\section{Introduction}

Knot invariants defined as combinations of integrals over configuration spaces or,
equivalently,
as combinations of algebraic counts of diagrams,
emerged after the seminal work of Witten \cite{[Wit]} on the perturbative
expansion of the Chern-Simons theory.
Knot invariants defined from spatial configurations of unitrivalent graphs were formally
defined by Guadagnini, Martellini and Mintchev \cite{GMM}, Bar-Natan \cite{BarNatanPCST}, Altschüler and Freidel \cite{Altschuler_1997},
Bott and Taubes \cite{Bott2017}, and others,
for knots in $\R^3$.
These invariants can be unified in an invariant $\mathcal Z= (\mathcal Z_k)_{k\in \mathbb N}$ defined by 
Altschüler and Freidel in \cite{Altschuler_1997} for knots in $\R^3$, and called the \emph{perturbative expansion of the Chern-Simons theory}.
The invariant $\mathcal Z$ takes its values in a vector space $\mathcal A= \prod\limits_{k\in\mathbb N} \mathcal A_k$, 
spanned by classes of Jacobi unitrivalent diagrams, 
precisely described in Definition \ref{ASSTU}. 
Altschüler and Freidel proved that $\mathcal Z$ is a universal Vassiliev invariant for knots in $\R^3$.
The Kontsevich integral $\mathcal Z^K=(\mathcal Z_k^K)_{k\in\mathbb N}$
described by Bar-Natan \cite{Bar-NatanTop}, and defined using integrals over spaces of planar configurations
is another universal Vassiliev invariant.
An article of Lescop \cite{LesJKTR} connects $\mathcal Z$ to $\mathcal Z^K$, up to the Bott and Taubes anomaly $\alpha$, 
and implies that the two invariants are equivalent.

In \cite{BNG},
Bar-Natan and Garoufalidis defined a linear form $w_C$ on $\bigoplus_{k\in\mathbb N}\mathcal A_k$ (see Definition \ref{wc0}), called the \emph{Conway weight system}, and they
expressed the Alexander polynomial $\Delta_\psi$ for knots in $\R^3$ as
\[ \Delta_\psi(e^h)=
\frac{2\sinh\left(\frac{h}2\right)}{h}
 \sum\limits_{k\geq0} (w_C \circ \mathcal Z^K_k)(\psi)h^k
.\]
Both hands of the formulas are also well-defined for the long knots\footnote{See Section \ref{S21} for a definition of long knots, in a wider setting.} in $\R^3$, and 
\[ \Delta_\psi(e^h)= \sum\limits_{k\geq0} (w_C \circ \mathcal Z^K_k)(\psi)h^k.\]

The perturbative expansion of Chern-Simons theory $(\mathcal Z_k)_{k\in\mathbb N}$ extends to long knots in rational asymptotic homology\footnote{These spaces are defined in Section \ref{S21}.} $\R^3$,
as in 
\cite{[Lescop],[Lescop2]}.
In this article, we prove the following result (Corollary
\ref{e2}).
\begin{thmnonnum}
For any null-homologous long knot $\psi$ of an asymptotic rational homology $\R^3$,
\[ \Delta_\psi(e^h)= \sum\limits_{k\geq0} (w_C \circ \mathcal Z_k)(\psi)h^k.\]
 
\end{thmnonnum}

In Proposition \ref{CSK}, we use the relation of Lescop \cite{LesJKTR} between $\mathcal Z$ and $\mathcal Z^K$ 
to notice that our theorem in terms of the perturbative expansion of the Chern-Simons theory
is equivalent to the formula of Bar-Natan and Garoufalidis in terms of the Kontsevich integral for long knots of $\R^3$.
The proof of our more general
theorem relies on completely different methods, even for long knots of $\R^3$, 
and our theorem holds in the wider setting of null-homologous long knots in asymptotic rational homology $\R^3$.
Our proof uses the direct computations of integrals over configuration spaces of our article \cite{article2}.
For codimension two null-homologous long knots of any 
asymptotic homology\footnote{The definition given for $n=1$ in Section \ref{S21} easily adapts to any $n\geq1$.} $\R^{n+2}$, these computations allowed us to express 
the generalized Bott-Cattaneo-Rossi (BCR for short) invariants $(Z_{BCR,k})_{k\in\mathbb N\setminus\{0,1\}}$, 
which we defined in \cite{article1},
in terms of 
the Reidemeister torsion (or Alexander polynomials).

These generalized BCR invariants generalize invariants defined by 
Bott \cite{[Bott]}, and Cattaneo and Rossi \cite{Cattaneo2005} for (codimension $2$) long knots in odd-dimensional Euclidean spaces 
$\R^{n+2}$ with $n+2\geq5$.
The BCR invariant $Z_{BCR,k}$ is a combination of integrals over configuration spaces
associated with some diagrams with $2k$ vertices of two kinds and $2k$ edges of two kinds, called \emph{BCR diagrams}.
In \cite{article1}, given a parallelized asymptotic homology $\R^{n+2}$ $(\spamb, \tau)$, 
we defined some particular forms on the two-point configuration spaces of $\R^n$ or of $\spamb$, 
called \emph{propagators} of $(\spamb, \tau)$.
The generalized BCR invariant $Z_{BCR,k}$ maps the data of 
a parallelized asymptotic homology $\R^{n+2}$ $(\spamb, \tau)$, a long knot $\psi$, and a family of propagators $F$ of $(\spamb, \tau)$ to a real number. 
In \cite{article1}, we proved that this number depends only of the diffeomorphism class of $(\spamb, \psi)$ when $n\geq3$.
For $n=1$, the definition of \cite{article1} still makes sense, but $Z_{BCR,k}$ might depend on the choice of 
the parallelization or of the propagators, and might not be an invariant.
In this article, we prove the formula
\begin{equation}\label{e3}
Z_{BCR,k}(\psi) = - \left(w'_C \circ \mathcal Z_k\right)(\psi),\end{equation}
for any long knot $\psi$ of a rational asymptotic homology $\R^3$ and
for the weight system $w'_{C}$ defined in Lemma \ref{wc}, where the invariant $Z_{BCR,k}$ can be computed with \emph{any} set of propagators. 
In particular, Formula \ref{e3} and the results of \cite{[Lescop2]} on the perturbative expansion of Chern-Simons theory imply that for any long knot $\psi$ of an asymptotic homology $\R^3$, 
the number $Z_{BCR,k}(\psi)$ does not depend on the choice of the propagators or of the parallelization,
and that it is invariant under ambient diffeomorphism : this is Corollary \ref{bonnedef}.
The proof of Formula \ref{e3} only relies on combinatorics of BCR diagrams and Jacobi unitrivalent diagrams.

The weight system $w'_C$ coincides with the Conway weight system $w_C$ defined by Bar-Natan and 
Garoufalidis in \cite{BNG} on the non-empty connected unitrivalent diagrams.
It vanishes on trivalent diagrams, and on non-trivial products of diagrams.
It satisfies the formula
\begin{equation}\label{e5}\sum\limits_{k\geq0}(w_C\circ\mathcal Z_k)(\psi) h^k= \exp\left(
\sum\limits_{k\geq1} (w'_C\circ \mathcal Z_k)(\psi) h^k \right),\end{equation}
for any long knot of an asymptotic rational homology $\R^3$, as noticed in Theorem \ref{th0}.

For any $n\geq1$, our flexible definition of \cite{article1} for the BCR invariants 
allows us to compute them with arbitrary propagators. In \cite{article2}, we present an
explicit computation based on so-called \emph{admissible propagators}, which
yields
exact formulas for $Z_{BCR,k}(\psi)$ in terms of Alexander polynomials.
For 	a null-homologous long knot of an asymptotic rational homology $\R^3$, these formulas reduce to
\begin{equation}\label{e4} \Delta_\psi(e^h)= \exp\left( -\sum\limits_{k\geq2} Z_{BCR,k}(\psi) h^k \right).\end{equation}
Formulas \ref{e3}, \ref{e5} and \ref{e4} imply the theorem stated in the beginning of this introduction.

In Section \ref{Partie1}, we review the definitions of the two invariants $\mathcal Z_k$ and $Z_{BCR, k}$ of this article, 
and we state the forementioned results in Theorems \ref{th1} (Formula \ref{e3} above) 
and \ref{th-1} (Formula \ref{e4} above). Section \ref{Partie2} is devoted to the proof of Theorem \ref{th1}.

I thank my advisor, Christine Lescop, for her help in the redaction of this article. 
\section{Bott-Cattaneo-Rossi invariants and perturbative expansion of the Chern-Simons theory}\label{Partie1}
\subsection{Long knots in asymptotic homology \texorpdfstring{$\R^3$}{R3}}\label{S21}
Let $\sphamb$ be a smooth oriented compact connected $3$-manifold with the rational homology of $\s^3$. 
Fix a point $\infty$ in $\sphamb$ and a closed ball $B_\infty(M)$ around $\infty$, 
and set $\spamb = \sphamb\setminus\{\infty\}$.
Identify the punctured ball $B_\infty^\circ(M) = B_\infty(M)\setminus\{\infty\}$ of $\spamb$ with the complement $\nbdinf$ of the open unit ball in $\R^3$, 
and let $\bM$ denote the closure of $\spamb\setminus \nbdinf$. 
The manifold $\spamb$ together with the decomposition $\spamb = \bM\cup \nbdinf$ is called an \emph{asymptotic rational homology $\R^3$}.

A \emph{parallelization} of an asymptotic rational homology $\R^3$ is a bundle isomorphism $\tau\colon \spamb \times \R^3\rightarrow T\spamb$ 
that coincides with the canonical trivialization of $T\R^3$ on $\nbdinf\times\R^3$. 
Such a parallelization always exists : see for example \cite[Proposition 5.5]{[Lescop2]}.

A \emph{long knot} in an asymptotic rational homology $\R^3$ is a smooth embedding $\psi \colon \R\hookrightarrow \spamb$ such that, for any $x\in \R$, \begin{itemize}
\item if $x\in[-1,1]$, $\psi(x)\in\bM$,
\item if $x\not\in [-1,1]$, $\psi(x) = (0,0,x)\in\nbdinf\subset \R^3$.
\end{itemize}

In the following, we let $(\spamb, \tau)$ be a fixed parallelized rational asymptotic homology $\R^3$.

\subsection{BCR diagrams}
We recall the definition of BCR diagrams, as introduced in \cite[Section 2.2]{article1}. In all the following, if $k$ is a positive integer, $\und{k}$ denotes the set $\{1,\ldots, k\}$.

\begin{df}\label{Def-BCR}
A \emph{BCR diagram} is a non-empty oriented connected graph $\Gamma$, defined by a set $\sommets$ of vertices, decomposed into $\sommets = \sommetsinternes\sqcup\sommetsexternes$, and a set $\aretes$ of ordered pairs of distinct vertices, decomposed into $\aretes = \aretesinternes\sqcup\aretesexternes$, whose elements are called \emph{edges}\footnote{Note that this implies that our graphs have neither loops nor multiple edges with the same orientation.}, where the elements of $\sommetsinternes$ are called \emph{internal vertices}, those of $\sommetsexternes$ \emph{external vertices}, those of $\aretesinternes$ \emph{internal edges}, and those of $\aretesexternes$ \emph{external edges}, and such that, for any vertex $v$ of $\Gamma$,  one of the five following properties holds: \begin{enumerate}
\item $v$ is external, with two incoming external edges and one outgoing external edge, 
and exactly one of the incoming edges comes from a univalent vertex. 
\item $v$ is internal and trivalent, with one incoming internal edge, one outgoing internal edge, and one incoming external edge, which comes from a univalent vertex.
\item $v$ is internal and univalent, with one outgoing external edge. 
\item $v$ is internal and bivalent, with one incoming external edge and one outgoing internal edge.
\item $v$ is internal and bivalent, with one incoming internal edge and one outgoing external edge.
\end{enumerate}
\end{df}
In the following, internal edges are depicted by solid arrows, external edges by dashed arrows, internal vertices by black dots, and external vertices by white dots, as in Figure \ref{BCR4}, where the five behaviors of Definition \ref{Def-BCR} appear.
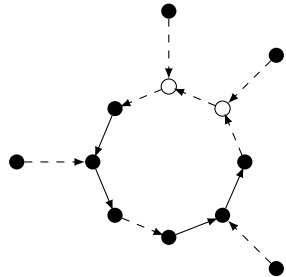
\begin{figure}[H]
\centering
\begin{tikzpicture}[xscale=-1]
\fill (0:1) \crc (45:1) \crc (0:2) \crc (90:2) \crc (135:2) \crc (180:1) \crc (225:1) \crc (225:2) \crc (270:1) \crc (315:1) \crc;
\draw (90:1) \crc (135:1) \crc;
\edgi (45:1) ++(-67.5:.1) -- ++(-67.5:.6); 
\edgi (0:1) ++(-112.5:.1) -- ++(-112.5:.6); 
\edge (-45:1) ++(-157.5:.1) -- ++(-157.5:.6); 
\edgi (-90:1) ++(-201.5:.1) -- ++(-201.5:.6); 
\edgi (-135:1) ++(-246.5:.1) -- ++(-246.5:.6); 
\edge (180:1) ++(67.5:.1) -- ++(67.5:.6); 
\edge (135:1) ++(22.5:.1) -- ++(22.5:.6); 
\edge (90:1) ++(-22.5:.1) -- ++(-22.5:.6); 
\edge (0:1.9) -- (0:1.1); 
\edge (90:1.9) -- (90:1.1); 
\edge (135:1.9) -- (135:1.1); 
\edge (-135:1.9) -- (-135:1.1); 
\end{tikzpicture}
\caption{An example of a BCR diagram of degree 6}\label{BCR4}
\end{figure}
Definition \ref{Def-BCR} implies that any BCR diagram consists of one cycle with some legs attached to it, where \emph{legs} are external edges that come from a (necessarily internal) univalent vertex, and where the graph is a cyclic sequence of pieces as in Figure \ref{fig-BCR2} with as many pieces of the first type than of the second type. 
In particular, a BCR diagram has an even number of vertices, 
and this number is also the number of its edges.

\begin{figure}[H]
\centering
\begin{tikzpicture}
\draw[white] (0,-1) circle (0.1);
\draw [->, dashed, >= latex] (-0.5, 0)-- (-0.1,0);
\fill (0,0) circle (0.1) ;
\draw  (.1, 0)-- (0.5,0);
\end{tikzpicture}\ \ \ \ \ \ \ \ \ \ 
\begin{tikzpicture}
\draw[white] (0,-1) circle (0.1);
\draw [->, >= latex] (-0.5, 0)-- (-0.1,0);
\fill (0,0) circle (0.1);
\draw [dashed] (.1, 0)-- (0.5,0);
\end{tikzpicture}\ \ \ \ \ \ \ \ \ \ 
\begin{tikzpicture}[yscale=-1]
\fill (1,0) circle (0.1) (3,0) circle (0.1);
\fill (1,1) circle (0.1) (3,1) circle (0.1);
\draw [->, dashed,>= latex] (1, 0.9)-- (1,0.1);
\draw [->,dashed,  >= latex] (3, 0.9)-- (3,0.1);
\draw [->, >= latex] (.5, 0)-- (0.9,0);
\draw  (1.1, 0)-- (1.7,0);
\draw [->, >= latex] (2.3, 0)-- (2.9,0);
\draw [dotted] (1.7, 0)-- (2.3,0);
\draw (3.1, 0)-- (3.5,0);
\end{tikzpicture}\ \ \ \ \ \ \ \ \ \ 
\begin{tikzpicture}[yscale=-1]
\draw (1,0) circle (0.1) (3,0) circle (0.1) ;
\fill (1,1) circle (0.1) (3,1) circle (0.1) ;
\draw [->,dashed, >= latex] (1, 0.9)-- (1,0.1);
\draw [->,dashed, >= latex] (3, 0.9)-- (3,0.1);
\draw [->,dashed, >= latex] (.5, 0)-- (0.9,0);
\draw [dashed] (1.1, 0)-- (1.7,0);
\draw [->, dashed,>= latex] (2.3, 0)-- (2.9,0);
\draw [dotted] (1.7, 0)-- (2.3,0);
\draw [dashed] (3.1, 0)-- (3.5,0);
\end{tikzpicture}
\caption{ }\label{fig-BCR2}
\end{figure}
The \emph{degree} of a BCR diagram is the integer $\deg(\Gamma) = \frac12\Card(\sommets)$. 
A \emph{numbering}\footnote{In this article, numberings are valued in $\und{3k}$ rather than in $\und{2k}$ as in \cite{article1} or \cite{article2}, and only the external edges are numbered.} of a degree $k$ BCR diagram is
an injection $\sigma\colon\aretesexternes\hookrightarrow \und{3k}$.

\subsection{Jacobi unitrivalent diagrams}

In this section, we recall the definition of unitrivalent diagrams, widely used in the theory of Vassiliev invariants. 

\begin{df}
A \emph{Jacobi diagram} is a graph $\Gamma$,
given by a set $\sommets$ of vertices, decomposed into $\sommets = \sommetsinternes\sqcup
\sommetsexternes$, a set $\aretes$ of unordered pairs of distinct vertices called \emph{edges}, such that the vertices of $\sommetsexternes$ are trivalent, 
the vertices of $\sommetsinternes$ are univalent, and 
the set $\sommetsinternes$ of univalent vertices is totally ordered.
%A \emph{Jacobi diagram} (with support $\R$) is a graph $\Gamma$ 
%given by a set $\sommets$ of vertices, which are all univalent or trivalent, 
%a set $\aretes$ of unoriented edges, 
%and a total order on the univalent vertices.

The \emph{degree} $\deg(\Gamma)$ of such a diagram $\Gamma$ is half its number of vertices. 
A \emph{numbering} of a degree $k$ Jacobi diagram 
is an injection $j\colon \aretes \hookrightarrow \und{3k}$.

An \emph{orientation} of a trivalent vertex $v$ is
the choice of a cyclic order on the three half-edges adjacent to $v$.
A \emph{vertex-orientation} of a Jacobi diagram $\Gamma$ is 
the choice of an orientation of any trivalent vertex. 
A \emph{vertex-oriented Jacobi diagram} is a Jacobi diagram together with a vertex orientation.

An \emph{edge-orientation} of a Jacobi diagram is the choice of an orientation for each edge. An \emph{edge-oriented Jacobi diagram} is a Jacobi diagram with a given edge-orientation.
A \emph{bioriented Jacobi diagram} is a Jacobi diagram with both a vertex-orientation and an edge-orientation.
\end{df}
In the following, Jacobi diagrams will be depicted 
by a planar immersion of $\Gamma\cup \R$, 
where $\R$ is a plain vertical line, on which the univalent vertices of $\Gamma$ lie, 
and the edges of $\Gamma$ are dashed lines. 
The order on the univalent vertices is given by the vertical direction on the plain vertical line $\R$ from bottom to top.
When dealing with vertex-oriented Jacobi diagrams,
the vertex-orientation will be given by the counterclockwise order in the plane.
An example of Jacobi diagram is given in Figure \ref{figure-Jacobi-ex}.

\begin{figure}[H]
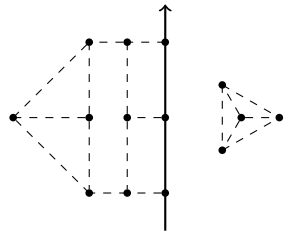

\centering
\figJAC
\caption{An example of a degree $7$ Jacobi diagram.}\label{figure-Jacobi-ex}
\end{figure}

\begin{df}\label{ASSTU}
Let $k$ be a nonnegative integer.
The space $\diagrammes$ is the real vector space spanned by 
the degree $k$ vertex-oriented Jacobi diagrams 
up to the equivalence relation spanned by the three following rules:
\begin{itemize}
\item \emph{AS relation:} if $\overline \Gamma$ is obtained from $\Gamma$ by reversing the orientation of one trivalent vertex,
$[\overline\Gamma]= -[\Gamma]$,
\item \emph{IHX relation:} if $\Gamma_1$, $\Gamma_2$ and $\Gamma_3$ can be represented by planar immersions that coincide outside a disk and are as in the second row of Figure \ref{relations} inside this disk, $[\Gamma_1]+[\Gamma_2]+ [\Gamma_3]=0$. 
\item \emph{STU relation:} if $\Gamma$, $\Gamma_1$ and $\Gamma_2$ can be represented by planar immersions that coincide outside a disk and are as in the third row of Figure \ref{relations} inside this disk, $[\Gamma]= [\Gamma_1]-[\Gamma_2]$.
\end{itemize}
We also denote the vector subspace of $\diagrammes$ generated 
by the classes of diagrams such that any connected component contains a univalent vertex by $\diagrammesc$.
Set 
$\mathcal A=\prod\limits_{k\geq0}\diagrammes$ and $\mathcal{\check A}=\prod\limits_{k\geq0}\diagrammesc$.
Note that $\mathcal A_0 = \mathcal{\check A}_0 = \R.[\emptyset]$, where $[\emptyset]$
is the class of the empty diagram.
\end{df}
\begin{figure}[h]
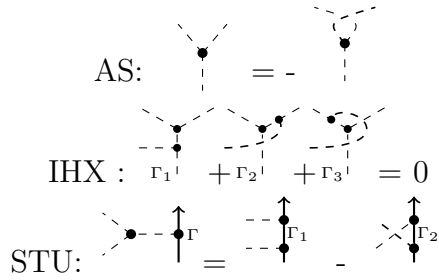

\centering
AS: \figureas\\
IHX : \figureihx \\
STU:\figurestu
\caption{The three relations of Definition \ref{ASSTU}.}\label{relations}
\end{figure}
\begin{df}\label{produit}
The \emph{product} of two Jacobi diagrams $\Gamma_1$ and $\Gamma_2$ is the Jacobi diagram $\Gamma$ such that  \begin{itemize}
\item the set of univalent vertices of $\Gamma$ is $V_i(\Gamma) = V_i(\Gamma_1)\sqcup V_i(\Gamma_2)$,
\item the set of trivalent vertices of $\Gamma$ is $V_e(\Gamma) = V_e(\Gamma_1)\sqcup V_e(\Gamma_2)$,
\item the set of edges of $\Gamma$ is $E(\Gamma) = E(\Gamma_1)\sqcup E(\Gamma_2)$,
\item the order on $V_i(\Gamma)$ is the unique order compatible with the injections $\big(V_i(\Gamma_j)\hookrightarrow V_i(\Gamma)\big)_{j=1,2}$ such that any element of $V_i(\Gamma_1)$ is before any element of $V_i(\Gamma_2)$. 
\end{itemize}
Note that the data of vertex-orientations (resp. edge-orientations) for $\Gamma_1$ and $\Gamma_2$ induce a natural
vertex-orientation (resp. edge-orientation) for their product $\Gamma$.
\end{df}

The product of Jacobi diagrams is compatible with the relations of Definition \ref{ASSTU}.
Bar-Natan proved in \cite[Theorem 7]{Bar-NatanTop} that the induced graded algebra structure on $\prod\limits_{k\in\mathbb N}\mathcal A_k$ is commutative.\footnote{Bar-Natan's theorem proves that $\prod\limits_{k\in\mathbb N}\mathcal A_k$ is a commutative and cocommutative Hopf algebra, for a given coproduct, but we do not use the coproduct in this article.}
This allows us to define the associated exponential map $\exp \colon \prod\limits_{k\in\mathbb N} \mathcal A_k \rightarrow \prod\limits_{k\in\mathbb N} \mathcal A_k$.

\subsection{The Conway weight system}

Let us first recall the definition of the Conway weight system of \cite[Section 3.1]{BNG}. 

\begin{df}
Let $\Gamma$ be a Jacobi diagram with only univalent vertices (such a diagram is called a \emph{chord diagram}). 
Use the edges of $\Gamma$ to do the surgeries on the line $\R$ as in Figure \ref{figwcBN}. 
The obtained manifold is the disjoint union of one line and $c$ circles. 
The \emph{Conway weight system} $w_C$ is defined as \[w_C(\Gamma)=
\begin{cases} 1 & \text { if $c=0$,} \\ 0 & \text{otherwise.}\end{cases}\]
\end{df}

\begin{figure}[H]
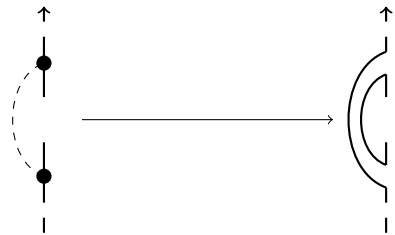

\centering
\wczero\\
\caption{Surgeries involved in the definition of $w_C$}\label{figwcBN}
\end{figure}

\begin{figure}[H]
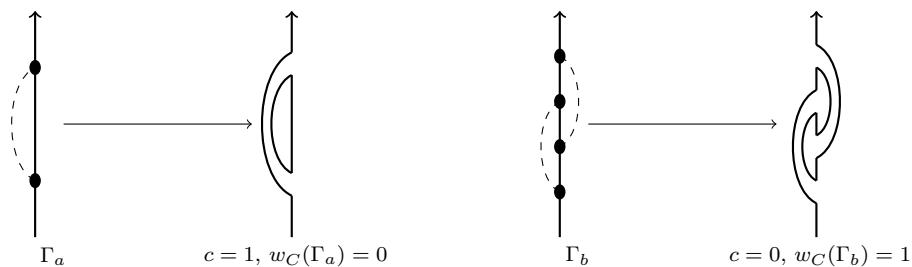

\centering
\wcun \ \ \ \ \ \ \ \ \ \ \ \wcdeux
\caption{Computation of $w_C(\Gamma)$ for two chord diagrams.}
\end{figure}
In \cite{BNG}, Bar-Natan and Garoufalidis proved that this definition determines a linear form $w_C\colon \bigoplus\limits_{k\in\mathbb N}\diagrammesc\rightarrow \R$. 
Set $w_C(\Gamma)=0$ for diagrams with at least one component without univalent vertices.
The form $w_C$ naturally extends to $w_C\colon \bigoplus\limits_{k\in\mathbb N}\diagrammes\rightarrow \R$
by sending diagrams with a non-empty trivalent\footnote{A \emph{trivalent graph} is a graph with only trivalent vertices.} connected
component to zero.
Chmutov \cite[p. 9]{[Chmutov]} proved that $w_C$ is determined by the following properties.

\begin{lm}\label{wc0}
The Conway weight system $w_C$ is the unique linear form $w_C\colon \bigoplus\limits_{k\in\mathbb N}\diagrammes\rightarrow \R$ such that 
\begin{itemize}
\item $w_C$ vanishes on $\mathcal A_1$ and maps $[\emptyset]$ to $1$,
\item for any integer $k\geq2$, if $\Gamma_k$ denotes the diagram depicted in Figure \ref{figgk}, $w_{C}([\Gamma_k])=-1-(-1)^k$,
\item if the number of trivalent vertices of $\Gamma$ is greater than its degree, then  $w_{C}([\Gamma])=0$,
\item if $\Gamma$ is the product of two diagrams $\Gamma_1$ and $\Gamma_2$, then
$w_C([\Gamma]) = w_C([\Gamma_1])w_C([\Gamma_2])$.
\end{itemize}
In particular, $w_C$ vanishes on odd-degree diagrams. 
\end{lm}\begin{figure}[H]
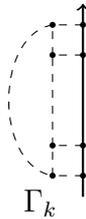
\centering
\graphek
\caption{The degree $k$ Jacobi diagram $\Gamma_k$}\label{figgk}
\end{figure}

The following result directly follows from \cite[Corollary 6.36]{[Lescop2]}.
\begin{lm}\label{cor27}
For any $k\geq 1$, let \begin{itemize}
\item $\primitifs$ denote the subspace of $\diagrammesc$ spanned by the classes of 
connected diagrams with at least one univalent vertex,
\item $\nonprimitifs$ denote the subspace of $\diagrammesc$ spanned by the classes of non-trivial products, which are
products of two non-empty diagrams,
\item $\diagrammess$ denote the subspace of $\diagrammes$ spanned by 
the classes of degree $k$ diagrams with at least one 
trivalent connected component,
\end{itemize}
and set $\mathcal P_0= \mathcal T_0= \{0\}$ and $\mathcal N_0=\R.[\emptyset]$.
For any $k\geq 0$, the space $\diagrammes$ splits into $\diagrammes = \primitifs \oplus \nonprimitifs\oplus\diagrammess$.
This yields a natural projection $p^c\colon\bigoplus\limits_{k\geq0}\diagrammes \rightarrow \bigoplus\limits_{k\geq0}\primitifs$.
\end{lm}

Define the \emph{logarithmic Conway weight system} $w'_C$ as $w'_C =  w_C\circ p^c$. It is 
characterized as follows.

\begin{lm}\label{wc}
The map $w'_C$ is the unique linear form $w'_{C}\colon \bigoplus\limits_{k\in\mathbb N}\diagrammes \rightarrow \R$ such that \begin{itemize}
\item $w'_C$ vanishes on $\mathcal A_0\oplus\mathcal A_1$,
\item for any $k\geq2$, $w'_{C}([\Gamma_k])=-1-(-1)^k$,
\item if the number of trivalent vertices of $\Gamma$ is greater than its degree, then $w'_{C}([\Gamma])=0$,
\item if $\Gamma$ is a non-trivial product of Jacobi diagrams, then $w'_{C}([\Gamma])=0$.
\end{itemize}
\end{lm}
\begin{proof}

By Lemma \ref{wc0}, $w'_C$ satisfies these
properties, and they characterize it on the summand $\bigoplus\limits_{k\in\mathbb N} \primitifs$ of
$\bigoplus\limits_{k\in\mathbb N}\diagrammes$.\qedhere
\end{proof}

\subsection{Configuration spaces}
\subsubsection{For BCR diagrams}\label{s241}

\begin{df}
Let $\psi\colon \R\hookrightarrow \spamb$ be a long knot. Let $\Gamma$ be a BCR diagram.
The open configuration space associated to $\Gamma$ and $\psi$ is 
\[C_{\Gamma,BCR}^0(\psi) = \{ c \colon \sommets \hookrightarrow \spamb \mid\text{There exists } c_i\colon\sommetsinternes\hookrightarrow\R, c_{|\sommetsinternes}= \psi\circ c_i\}.\]
\end{df}
This space is oriented as follows. 
Let $\Gamma$ be a BCR diagram. 
For any internal vertex $v$, let $\d t_v$ denote the coordinate $c_i(v)$.
For any external vertex $v$, let $(\d X_v^{i})_{i\in\{1,2,3\}}$ denote the coordinates of $c(v)$ in an oriented chart of $\spamb$.
Split any external edge $e$ into two half-edges $e_-$ (the tail) and $e_+$ (the head). 
For any external half-edge $e_\pm$, 
define a form $\Omega_{e_\pm}$ as follows:
 \begin{itemize}
\item for the head $e_+$ of an edge that is not a leg, going to an external vertex $v$, $\Omega_{e_+} = \d X_v^1$,
\item for the head $e_+$ of a leg going to an external vertex $v$, $\Omega_{e_+} = \d X_v^2$,
\item for the tail $e_-$ of an edge coming from an external vertex $v$, $\Omega_{e_-} = \d X_v^3$,
\item for any external half-edge $e_\pm$ adjacent to an internal vertex $v$, $\Omega_{e_\pm} = \d t_v$,
\end{itemize}
and set $\epsilon(\Gamma) = (-1)^{\Card(\aretesexternes)+N_T(\Gamma)}$ where $N_T(\Gamma)$ is the number of trivalent vertices of $\Gamma$. 
With these notations, the manifold $C_{\Gamma,BCR}^0(\psi)$ is oriented by the form $\Omega(\Gamma) = \epsilon(\Gamma) \bigwedge\limits_{e\in\aretesexternes}\left(\Omega_{e_-}\wedge\Omega_{e_+}\right)$.

\subsubsection{For Jacobi diagrams}
\begin{df}
Let $\psi\colon \R\hookrightarrow \spamb$ be a long knot. Let $\Gamma$ be a Jacobi diagram.
The open configuration space associated to $\Gamma$ and $\psi$ is 
\[C_{\Gamma,J}^0(\psi) = \left\{ c \colon \sommets \hookrightarrow \spamb ~ \Bigr| ~ \begin{array}{ll}
\text{There exists an increasing map } c_i\colon\univalents\hookrightarrow\R\\ \text{such that } c_{|\sommetsinternes}= \psi\circ c_i\end{array}
\right\}.\]
\end{df}

Let us fix a bioriented Jacobi diagram $\Gamma$
and orient the manifold $C_{\Gamma,J}^0(\psi)$. 
We use the edge-orientation to split any edge $e$ of $\Gamma$ in two half-edges $e_-$ and $e_+$ as above and, 
for each trivalent vertex, 
we fix an order on 
the set of its three adjacent half-edges, 
compatible with the cyclic order given by the vertex-orientation of 
the diagram.
For any half-edge $e_\pm$ of $\Gamma$, define forms $\Omega_{e_\pm}$ such that,
\begin{itemize}
\item if $e_\pm$ is adjacent to a univalent vertex $v$, $\Omega_{e_\pm} = \d t_v$,
\item if $e_\pm$ is the $i$-th half-edge adjacent to a trivalent vertex $v$, $\Omega_{e_\pm} = \d X_v^i$,
\end{itemize}
and set $\Omega(\Gamma) = \bigwedge\limits_{e\in E(\Gamma)} \left(\Omega_{e_-}\wedge\Omega_{e_+}\right)$.
 Note that this form does not depend of the "compatible with the vertex-orientation" choice of the orders of the half-edges around each trivalent vertex. 
%Therefore, $\Omega(\Gamma)$ only depends on the bi-oriented Jacobi diagram $\Gamma$.
Note that the orientation form $\Omega(\Gamma)$ is multiplied by $-1$
if we change the orientation of one edge or of one trivalent vertex.

\subsection{Propagators and configuration space integrals}

Here, we review the definition of $\configM$ in \cite[Section 2.2]{[Lescop]}.
Let $\configM$ denote the space obtained from $M^2$ after the differential blow-up \footnote{For example, see \cite[Section 2.2]{[Lescop]} for more details on these blow-ups.} of $\{(\infty,\infty)\}$ and 
of the closures of $\spamb\times\{\infty\}$, $\{\infty\}\times\spamb$ and $\Delta = \{(x,x)\mid
x\in\spamb\}$ in the obtained manifold. 
The manifold with boundary and corners $\configM$ is a compactification of $C_2^0(\spamb)=(\spamb)^2\setminus\Delta$.
Let $G_\tau\colon \partial\configM\rightarrow \s^2$ denote the Gauss map as defined 
in\footnote{It is called $p_\tau$ in these sources.} \cite[Proposition 2.3]{[Lescop]} or \cite[Proposition 3.7]{[Lescop2]}.
The map $G_\tau$ is an analogue of the map $G\colon C_2(\R^3)\rightarrow \s^2$ that extends
$\big((x,y)\in C_2^0(\R^3) \mapsto \frac{y-x}{||y-x||}\in\s^2\big)$, 
but $G_\tau$ is only defined on the boundary of the two-point configuration space.
It depends on the parallelization $\tau$ of Section \ref{S21}. 

A form $\Prop$ on $\configM$ is \emph{antisymmetric} if $T^*(\Prop)=-\Prop$, 
where $T\colon\configM\rightarrow\configM$ is the smooth extension of $(x,y)\in C_2^0(\spamb)\mapsto (y,x)\in\configM$.
Let us recall the definition of the external propagators of \cite[Section 2.5]{article1}. 
\footnote{Here, we require antisymmetric propagators in order to get a simpler formula in Theorem \ref{th0}.}
\begin{df}\label{prop}
A \emph{propagator} of $(\spamb,\tau)$ is a closed antisymmetric $2$-form $\Prop$ on $\configM$ such that
there exists a closed $2$-form $\delta_\Prop$ on $\s^2$
with total volume $1$ such that 
$\Prop_{|\partial\configM} = {G_\tau}^*(\delta_\Prop)$.

%A \emph{homogeneous} propagator is a propagator $\Prop$ such that $\delta_\Prop$ is the $SO(3)$-invariant form of total volume $1$ on $\s^2$.
%We say that a propagator of $(\spamb, \tau)$ is \emph{good} (for $(\psi)$ if it is homogeneous, 
%or if $\tau_{\psi(x)}(0,0,1) = \psi'(1)$ for any $x\in \R$.
A $k$-family of propagators of $(\spamb,  \tau)$ is the data of $3k$ propagators $(\Prop_i)_{i\in\und{3k}}$ of $(\spamb, \tau)$.
\end{df}
In $\R^3$ with its canonical parallelization, 
the pull-back of the $SO(3)$-invariant form on $\s^2$ with total volume $1$ under the Gauss map $G$ is a propagator. 

\begin{df}
For any edge-oriented Jacobi (resp. BCR) diagram $\Gamma$, define the following maps on the configuration space $C_{\Gamma, J}^0(\psi)$ (resp. $C_{\Gamma, BCR}^0(\psi)$):
\begin{itemize}
\item for any edge (resp. external edge) $e=(v,w)$ of $\Gamma$, $p_e$ denotes the map 
$\big(c\mapsto (c(v), c(w))\big)$ from the configuration space to $\configM$,
\item for any pair $(v,w)$ of distinct univalent (resp. internal) vertices, $\epsilon_{v,w}$ denotes 
the map $\big( c\mapsto \mathrm{sign}(c_i(w)-c_i(v))\big)$
from the configuration space to $\{-1,1\}$.
\end{itemize}
Note that the map $p_e$ depends on the chosen edge-orientation of 
a Jacobi diagram.
\end{df}

For any degree $k$ numbered BCR diagram $(\Gamma,\sigma)$
and any $k$-family $F=(\Prop_i)_{i\in\und{3k}}$ of propagators of $(\spamb, \psi, \tau)$,
define a form $\omega^F(\Gamma,\sigma,\psi)$ on $C^0_{\Gamma,BCR}(\psi)$
as
\[\omega^F(\Gamma,\sigma,\psi) = \frac{(-1)^{N_i^-(\Gamma,\cdot)}}{2^{\Card(\aretesinternes)}}\bigwedge\limits_{e\in\aretesexternes} {p_e}^*(\Prop_{\sigma(e)}),\] where for any $c\in \confignoeud$, $N_i^-(\Gamma,c)$ is the number of internal edges from a vertex $v$ to a vertex $w$ such that
$\epsilon_{v,w}(c)<0$. 
Note that the form $\omega^F(\Gamma,\sigma,\psi)$ is the analogue of the form of the same name defined in \cite[Section 2.6]{article1}, 
when all the internal propagators of this article are equal to the $0$-form $\alpha$ on $C_2(\R)$ that extends $\left((x,y)\in C_2^0(\R)\mapsto  \frac{\sgn(y-x)}2\in \R\right)$.

For any degree $k$ numbered edge-oriented Jacobi diagram $(\Gamma, j)$ 
and any $k$-family $F=(\Prop_i)_{i\in\und{3k}}$ of propagators of $(\spamb, \tau)$,
define a form $\omega^F(\Gamma,j,\psi)$
as
\[\omega^F(\Gamma,j,\psi) = \bigwedge\limits_{e\in\aretes} {p_e}^*(\Prop_{j(e)}).\]
Now, for any numbered BCR diagram, set \[
I^F(\Gamma,\sigma,\psi) = \int_{C_{\Gamma,BCR}^0(\psi)} \omega^F(\Gamma,\sigma, \psi),\]
and for any numbered bioriented Jacobi diagram, set \[
I^F(\Gamma,j,\psi) = \int_{C_{\Gamma,J}^0(\psi)} \omega^F(\Gamma,j, \psi).\]
These two integrals converge because of the existence of compactifications of 
the configuration spaces, to which the above forms extend.\footnote{Such compactifications were first introduced by Axelrod and Singer in \cite[Section 5]{AS} for Jacobi diagrams and by Rossi in his thesis \cite[Section 2.5]{[Rossi]} for BCR diagrams. See also \cite[Section 8]{[Lescop2]} for an extensive description of these compactifications for Jacobi diagrams.}

\subsection{The perturbative invariant \texorpdfstring{$\mathcal Z$}{Z}}

Let $\graphesJn$ denote the set of degree $k$ (unoriented) numbered Jacobi diagrams.
Let $(\Gamma, j)\in \graphesJn$, and let $F= (\omega_i)_{i\in\und{3k}}$
be a family of propagators of $(\spamb, \tau)$.
Orient both the edges and the trivalent vertices of $\Gamma$ arbitrarily.
Since the propagators are antisymmetric,
the orientation of $C_{\Gamma,J}^0(\psi)$ and the sign of $\omega^F(\Gamma,j, \psi)$
both change when changing the orientation of one edge of $\Gamma$, and $w'_C(\Gamma)$ remains unchanged. 
When the orientation of one trivalent vertex changes,
the sign of $w'_C([\Gamma])$ and the orientation of $C_{\Gamma,  J}^0(\psi)$ both change and $\omega^F(\Gamma,j,\psi)$ remains unchanged.
Therefore, $I^F(\Gamma,j, \psi)w'_C([\Gamma])$ depends neither of the vertex-orientation, nor of the edge-orientation, and is thus well-defined for $(\Gamma,j)\in\graphesJn$.

Let $j_\emptyset$ denote the only numbering of the empty diagram, 
let $F_\emptyset$ denote the empty family of propagators, and set  $I^{F_\emptyset}(\emptyset,j_\emptyset, \psi)= 1$. 

\begin{theo}\label{th0}
Let $(\spamb, \tau)$ be a parallelized asymptotic rational homology $\R^3$. Fix a long knot $\psi$ of $\spamb$. 
Fix an integer $k\geq0$, and a $k$-family $F=(\Prop_i)_{i\in\und{3k}}$ of propagators of $(\spamb, \tau)$, and set \[\left(w_C\circ \mathcal Z^F_{k}\right)(\psi, \tau) =  \sum\limits_{(\Gamma,j)\in\graphesJn} \frac{(3k-\Card(\aretes))!}{(3k)!} I^F(\Gamma,j, \psi)w_C([\Gamma])
,\]
\[\left(w'_C\circ \mathcal Z^F_{k}\right)(\psi, \tau) =  \sum\limits_{(\Gamma,j)\in\graphesJn} \frac{(3k-\Card(\aretes))!}{(3k)!} I^F(\Gamma,j, \psi)w'_C([\Gamma])
.\]
\begin{itemize}
\item The quantity $w'_C\circ \mathcal Z_{k}^F$ does not depend on the choice of the $k$-family $F$ of propagators of $(\spamb,\tau)$.
\item The quantity $w'_C\circ \mathcal Z_{k}=w'_C\circ \mathcal Z_{k}^F$ does not depend on the choice of the parallelization $\tau$ of $\spamb$.
\item The quantity $w'_C\circ \mathcal Z_{k}$ only depends on the diffeomorphism class of $(\spamb,\psi)$.
\item We have the following equality in $\R[[h]]$: $$\sum\limits_{k\geq0}(w_C\circ\mathcal Z_k)(\psi) h^k= \exp\left(
\sum\limits_{k>0} (w'_C\circ \mathcal Z_k)(\psi) h^k \right).$$
\end{itemize}
\end{theo}
\begin{proof}
\cite[Theorem 12.32]{[Lescop2]}, which is a mild generalization for long knots of \cite[Theorem 7.20]{[Lescop2]}, implies that 
\[Z_{k}^F(\psi,\tau) = \sum\limits_{(\Gamma,j)\in\graphesJn} \frac{(3k-\Card(\aretes))!}{(3k)!} I^F(\Gamma,j, \psi)[\Gamma]\] 
is independent of the choice of the family $F$ of propagators of $(\spamb,  \tau)$.
\cite[Theorem 12.32]{[Lescop2]}, the last assertion of \cite[Theorem 12.18]{[Lescop2]} and \cite[Theorem 6.37]{[Lescop2]} imply that
$$\mathfrak Z_k(\psi)= p^c\circ Z_k^F(\psi, \tau) -\frac14 p_1(\tau) \beta_k - I_\theta(\psi, \tau) \alpha_k$$
does not depend on the parallelization and is invariant under diffeomorphism, 
where $\frac14 p_1(\tau)$ and $ I_\theta(\psi, \tau)$ are some real numbers, 
and where $\alpha_k$ and $\beta_k$ are two elements
of $\diagrammes$, called anomalies. \cite[Propositions 10.14 \& 10.20]{[Lescop2]} imply
that $\alpha_k$ and $\beta_k$ vanish if $k$ is even. Since $w_C$ vanishes on odd-degree
diagrams, this proves that $w'_C \circ \mathcal Z_k^F$ as defined in the theorem coincides 
with $w'_C \circ \mathfrak Z_k$. In particular, this implies the first three
assertions of the theorem.

%
%\textbf{Version 2 :}
%
%The first three assertions follow from \cite[Theorem 7.30]{[Lescop2]}\footnote{This theorem applies to "closed" knots $\s^1\hookrightarrow \spamb$, but all the arguments work for long knots $\R\hookrightarrow \spamb$ as here.}.
%
%\textbf{De toute manière, pour finir : }
Set $\mathcal Z = \sum\limits_{k\geq0} \mathcal Z_k h^k$ in $\mathcal A[[h]]$.
Since $w_C$ is multiplicative, 
$\exp\circ w_C = w_C \circ \exp$. The last assertion of \cite[Theorem 12.18]{[Lescop2]} and \cite[Theorem 6.37]{[Lescop2]} imply that
$\exp\circ p^c\circ\mathcal Z = \mathcal Z$. Therefore, 
\[
w_C\circ \mathcal Z = w_C \circ \exp \circ p^c \circ \mathcal Z
= \exp\circ w_C\circ p^c \circ \mathcal Z
= \exp\circ w'_C \circ \mathcal Z.\qedhere\]
\end{proof}

\subsection{The BCR invariants}
The following theorem is the main result of this article.
It is proved in Section \ref{Partie2}.
\begin{theo}\label{th1}
Fix an integer $k\geq2$, a null-homologous long knot $\psi$ of a parallelized asymptotic 
rational homology $\R^3$ $(\spamb, \tau)$, and a $k$-family $F=(\Prop_i)_{i\in\und{3k}}$ of propagators of $(\spamb, \tau)$.
Set 
\[ Z^F_{BCR, k}(\psi, \tau) =  \sum\limits_{(\Gamma,\sigma)\in\graphesnum} 
\frac{(3k-\Card(\aretesexternes))!}{(3k)!}
 I^F(\Gamma,\sigma, \psi)
 ,\]
where $\graphesnum$ denote the set of degree $k$ numbered BCR diagrams up to numbered graph isomorphisms.
We have \[
Z_{BCR, k}^F(\psi, \tau)= - (w'_C\ \circ\mathcal  Z_{k})(\psi)
.\]
\end{theo}
Note that the coefficients $\frac{(3k-\Card(\aretesexternes))!}{(3k)!}$ in the definition of 
$Z^F_{BCR, k}$ replace the $\frac1{(2k)!}$ in 
the definition of \cite[Theorem 2.10]{article1}
since we allowed numberings to take value in $\und{3k}$ and since the internal edges are not
numbered anymore.

The above theorem and Theorem \ref{th0} imply the following corollary. 
\begin{cor}\label{bonnedef}
The quantity $Z_{BCR,k}(\psi)=Z_{BCR,k}^F(\psi,\tau)$ does not depend on the choice of the propagators of $(\spamb,  \tau)$, 
nor of the parallelization, and is invariant under ambient diffeomorphism.
\end{cor}
\subsection{Relation with the Alexander polynomial}
Since the propagators defined in \cite[Section 4]{article2} are dual to propagators in the sense of Definition \ref{prop},
\cite[Theorem 2.31]{article2} implies the following theorem.
\begin{theo}\label{th-1}
Let $\spamb$ be an asymptotic rational homology $\R^3$, 
and let $\psi\colon\R\hookrightarrow \spamb$ be a null-homologous long knot of $\spamb$. 
If $\Delta_\psi(t)$ denotes the Alexander polynomial of $\psi$, then  
\[ \Delta_\psi(e^h)= \exp\left( -\sum\limits_{k\geq2} Z_{BCR,k}(\psi) h^k \right). \]
\end{theo}

The above theorem and Theorems \ref{th0} and \ref{th1} imply the following.
\begin{cor}\label{e2}
If $\spamb$ is an asymptotic rational homology $\R^3$, and if $\psi\colon\R\hookrightarrow \spamb$ is a null-homologous long knot of $\spamb$, then 
\[
\Delta_\psi(e^h) = \exp\left( \sum\limits_{k>0}(w'_C\circ\mathcal Z_k)(\psi) h^k\right)= \sum\limits_{k\geq0}(w_C\circ\mathcal Z_k)(\psi) h^k.
\]
\end{cor}

\subsection{Compatibility with the formula in terms of the Kontsevich integral}

In this section, we prove that Corollary \ref{e2} can be stated equivalently in terms of the Kontsevich integral
or of the perturbative expansion of the Chern-Simons theory, for long knots of $\R^3$. 

\begin{prop}\label{CSK}
For any long knot $\psi$ of $\R^3$,
\[
(w_C \circ \mathcal Z_k)(\psi)=(w_C \circ \mathcal Z^K_k)(\psi)
\]
\end{prop}

\begin{proof}
Let $\gamma = (\gamma_k)_{k\in \mathbb N\setminus\{0\}}$, where $\gamma_k$ is a combination of degree $k$ diagrams 
with exactly two univalent vertices,
and $\gamma_1\neq0$.
Lescop \cite[Definition 2.1]{LesJKTR} proved the existence of a well-defined morphism 
$\Psi(\gamma)\colon\mathcal A\rightarrow \mathcal A$
such that for any Jacobi diagram $\Gamma$,
$\Psi(\gamma)([\Gamma])$ is obtained as follows : 
\begin{itemize}
\item Write $\gamma_k = \sum\limits_{i\in I_k} b(\Gamma_{i,k}) [\Gamma_{i,k}]$ for any $k\geq1$, and set 
$I= \{\Gamma_{i,k}\mid k \geq1,  i \in I_k\}$.
\item In each connected component $\Gamma_j$ of $\Gamma$, fix
$\deg(\Gamma_j)$ edges, and let $X$ denote the set of the chosen edges.
\item For any map $\xi \colon X\rightarrow I$,
let $\Psi^0(\gamma)(\Gamma, \xi)$ denote the graph obtained from $\Gamma$
after replacing each edge $e$ of $X$ with $\xi(e)$.
\item Set $$\Psi(\gamma)([\Gamma])
= \sum\limits_{\xi \colon X\rightarrow I}\left( \prod\limits_{e\in X}b(\xi(e))\right)[\Psi^0(\gamma)(\Gamma, \xi)].
$$
\end{itemize}
With these notations, \cite[Theorem 2.3]{LesJKTR} states
that $\mathcal Z = \Psi(2\alpha)(\mathcal Z^K)$ where $\alpha$ is the Bott-Taubes anomaly.
\cite[Proposition 10.20]{[Lescop2]} implies that $2\alpha_1$ is the class of the diagram $\Gamma_\theta$ with one edge between two univalent vertices.

Proposition \ref{CSK} follows from the following lemma. 
\end{proof}

\begin{lm}
$w_C \circ \Psi(2\alpha) = w_C.$
\end{lm}
\begin{proof}\belowdisplayskip=-12pt
Set $w= w_C \circ \Psi(2\alpha)$, and let us check that $w$ 
satisfies the four properties of Lemma \ref{wc0}.
\begin{itemize}
\item The first property is immediate.

\item For any $k\geq3$, $2\alpha_k$ is a combination of diagrams with $2k-2>k$ trivalent vertices. 
Therefore, for any $p\geq2$ and any $\xi \colon X \rightarrow I$ as above,
if $\xi$ maps at least one edge to an element different from $\Gamma_\theta$,
then the number of trivalent vertices of $\Psi^0(2\alpha)(\Gamma_p, \xi)$ is greater than its degree.
This implies that $w([\Gamma_p])= w_C([\Gamma_p])$.

\item If $\Gamma$ contains more than $\deg(\Gamma)$ trivalent vertices,
then the same argument implies that $w([\Gamma])=0$.

\item If $\Gamma$ is the product of two non-empty diagrams $\Gamma_1$ and $\Gamma_2$, set $X_i= X\cap E(\Gamma_i)$ for $i\in\{1,2\}$. 
For any $\xi\colon X\rightarrow I$, set $\xi_i= \xi_{|X_i}$ for $i\in \{1,2\}$.
Note that $\Psi^0(2\alpha)(\Gamma, \xi)$ is the product of $\Psi^0(2\alpha)(\Gamma_1,\xi_1)$ and $\Psi^0(2\alpha)(\Gamma_2, \xi_2)$. Therefore, since $w_C$ is multiplicative, 
\begin{eqnarray*}
w([\Gamma]) &=& \sum\limits_{\xi\colon X\rightarrow I} 
\left(\prod\limits_{e\in X} b(\xi(e))\right)
w_C\left(\left[ \Psi^0(2\alpha)(\Gamma,\xi)\right] \right)\\
&=& \sum\limits_{\xi_1\colon X_1\rightarrow I} 
\sum\limits_{\xi_2\colon X_1\rightarrow I} 
\left(\prod\limits_{e\in X_1} b(\xi_1(e))\right)
\left(\prod\limits_{e\in X_2} b(\xi_2(e))\right)\\
& & \ \ \ \ \ \ \ \ 
w_C\left( \left[\Psi^0(2\alpha)(\Gamma_1,\xi_1)\right] \right)
w_C\left(\left[ \Psi^0(2\alpha)(\Gamma_2,\xi_2)\right] \right)\\
&=& w([\Gamma_1]) w([\Gamma_2]).
\end{eqnarray*}
\qedhere
\end{itemize}
\end{proof}
\section{Proof of Theorem \texorpdfstring{\ref{th1}}{2.14}}\label{Partie2}
\subsection{Definition of a weight system for the BCR invariants}

From now on, we fix a parallelized asymptotic rational homology $\R^3$
$(\spamb, \tau)$, a long knot $\psi\colon \R^n\hookrightarrow \spamb$, 
an integer $k\geq2$, 
and a $k$-family $F=(\Prop_i)_{i\in\und{3k}}$ of propagators of $(\spamb, \tau)$.

For any BCR diagram $\Gamma$, let $\ordres$ denote the set of total orders on $\sommetsinternes$.
Represent an element of $\ordres$ by the unique increasing bijection $\rho \colon \sommetsinternes\rightarrow \{1, \ldots,\Card(\sommetsinternes)\}$.
For any degree $k$ BCR diagram $\Gamma$, the configuration space 
$C_{\Gamma,BCR}^0(\psi)$ splits into a disjoint union of connected components 
$\bigsqcup\limits_{\rho\in\ordres} C_{\Gamma,BCR,\rho}^0(\psi)$ where 
\[C_{\Gamma,BCR,\rho}^0(\psi)= \left\{c \in C_{\Gamma,BCR}^0(\psi) ~ \Bigr| ~ 
\begin{array}{l}
\text{for any distinct internal vertices $v$, $w$, }\\ \rho(v)<\rho(w) \Leftrightarrow
\epsilon_{v,w}(c)>0 \end{array}
\right\}.\]

Any $\rho\in\ordres$ induces a bioriented Jacobi unitrivalent diagram $\Gamma_\rho$ with the same degree as follows. 
The vertices of $\Gamma_\rho$ are all the vertices of $\Gamma$
and the edges of $\Gamma_\rho$ are the external edges of $\Gamma$, so that 
the internal (resp. external) vertices of $\Gamma$ yield the univalent (resp. trivalent) vertices of $\Gamma_\rho$. 
The order of the univalent vertices of $\Gamma_\rho$ is given by $\rho$.
It remains to orient the trivalent vertices of $\Gamma_\rho$, 
i.e. to fix a cyclic order on the half-edges adjacent to any external vertex $v$. 
Let $e$ denote the external edge of the cycle going to $v$, 
let $\ell$ denote the leg going to $v$, 
and let $f$ denote the external edge of the cycle coming from $v$.
The orientation of $v$ is given by the cyclic order $( e_+,\ell_+, f_-)$.
Note that any numbering $\sigma$ of $\Gamma$ yields a canonical numbering $j_{\rho,\sigma}$ of $\Gamma_\rho$.
With these notations, the following lemma is immediate.

\begin{lm}\label{lmint}
For any BCR diagram $\Gamma$, any order $\rho\in\ordres$, and any numbering $\sigma$ of $\Gamma$,
\[
\int_{C_{\Gamma,BCR, \rho}(\psi)} \omega^F(\Gamma,\sigma, \psi) = \frac{\epsilon(\Gamma)\epsilon_2(\Gamma,\rho)}{2^{\Card(\aretesinternes)}}
I^F(\Gamma_\rho,j_{\rho,\sigma},\psi)
,\]
where $\epsilon(\Gamma)$ is defined in Section \ref{s241} and $\epsilon_2(\Gamma,\rho) =
\prod\limits_{(v,w)\in\aretesinternes} \sgn(\rho(w)-\rho(v))\in\{\pm1\}.$ 
%where $\sgn(\rho(w)-\rho(v))=\pm 1$ for any internal edge $(v,w)$ of $\Gamma$.
\end{lm}
Let us introduce the following notations.
%The above immediate lemma allows us to obtain the following expression of $Z_{k,BCR}^F(\psi, \tau)$.

\begin{nt}
For any degree $k$ vertex-oriented numbered Jacobi diagram $\Gamma_J$, let 
$\mathcal G(\Gamma_J,j)$ denote the set of ordered and numbered BCR diagrams $(\Gamma, \sigma, \rho)$ 
such that $(\Gamma_\rho, j_{\rho, \sigma})$ coincides with $(\Gamma_J, j)$, 
up to the orientation of the trivalent vertices and after forgetting the orientation of the edges.
For such a numbered and ordered BCR diagram, let $\epsilon_3(\Gamma_J,\Gamma, \rho)\in\{\pm1\}$ 
be such that $[\Gamma_\rho]= \epsilon_3(\Gamma_J,\Gamma, \rho)[\Gamma_J]$.
Set
\[w_{BCR}(\Gamma_J, j) = \sum\limits_{(\Gamma, \sigma, \rho)\in\mathcal G(\Gamma_J,j)} \frac{\epsilon(\Gamma)\epsilon_2(\Gamma,\rho)\epsilon_3(\Gamma_J,\Gamma,\rho)}{2^{2k-\Card(E(\Gamma_J))}}
,\] and note that $w_{BCR}(\Gamma_J, j)$ does not depend on $j$. Denote it by $w_{BCR}(\Gamma_J)$.
\end{nt}
We are going to prove the following proposition. 

\begin{prop}\label{th2}
For any Jacobi diagram $\Gamma_J$, $w_{BCR}(\Gamma_J) = - w'_C([\Gamma_J])$.
\end{prop}

Note that the above proposition and Lemma \ref{lmint} imply Theorem \ref{th1}.
We now prove Proposition \ref{th2} until the end of this article.
In the next subsections, we check that $(-w_{BCR})$ induces a linear map $\diagrammes\rightarrow \R$ that satisfies the properties of Lemma \ref{wc}. Note that it is immediate that $(-w_{BCR})$ vanishes on degree $0$ and $1$ diagrams, since a BCR diagram is non-empty, 
and since the only degree $1$ BCR diagram is counted with opposite signs 
when the order of its two internal vertices is changed.
The following lemma gives an example of diagrams with non-trivial $w_{BCR}$ and proves the second and third property of Lemma \ref{wc}.

\begin{lm}\label{calculgk}
\begin{enumerate}
\item If $\Gamma_J$ is a degree $k$ Jacobi diagram and has more than $k$ trivalent vertices, then $w_{BCR}(\Gamma_J)=0$.
\item If $\Gamma_{k}$ is the graph of Figure \ref{figgk}, then $w_{BCR}(\Gamma_k)=1+(-1)^k$.
\end{enumerate}
\end{lm}
\begin{proof}
The first point is immediate since 
any external vertex of a BCR diagram has a univalent neighbour, 
so that there cannot be more than $k$ external vertices in a degree $k$ BCR diagram.

For the second point, given a numbering $j$ of $\Gamma_k$,
there are exactly two elements in $\lettre(\Gamma_k, j)$, and 
they are given by $(\Gamma_k^{BCR}, \rho_a, j_a)$ and $(\Gamma_k^{BCR},\rho_b, j_b)$,
as depicted in Figure \ref{fig2}, where $v_i$ denotes the vertex $\rho^{-1}(i)$, and the numberings $j_a$ and $j_b$ are uniquely determined by $j$.

\begin{figure}[H]
\centering
\graphekA \ \ \graphekB \ \ \ \ \graphek
\caption{}
\label{fig2}
\end{figure}

Note that $\epsilon_2(\Gamma_k^{BCR},\rho_a)=\epsilon_2(\Gamma_k^{BCR},\rho_b)=1$, and that 
$\epsilon(\Gamma_k^{BCR}) = (-1)^k$.
The graph $(\Gamma_k^{BCR})_{\rho_a}$ has exactly the same vertex-orientation than $\Gamma_k$, 
and $(\Gamma_k^{BCR})_{\rho_b}$ has the opposite orientation at each trivalent vertex, so that
$(-1)^k[(\Gamma_k^{BCR})_{\rho_b}]=[(\Gamma_k^{BCR})_{\rho_a}]=[\Gamma_k]$.
Therefore, $w_{BCR}(\Gamma_k) = 1+(-1)^k $.
\end{proof}

\subsection{Linear extension of $w_{BCR}$ to $\diagrammes$}
\begin{lm}\label{STU}
The map $w_{BCR}$ induces a linear form on $\diagrammes$.
\end{lm}
\begin{proof}
It suffices to prove that $w_{BCR}$ is compatible with the relations of Definition \ref{ASSTU}. 
The compatibility with the AS relation is immediate. 
For diagrams with only trivalent vertices, the IHX relation is immediate since $w_{BCR}$ is zero on the three involved diagrams.
Bar-Natan \cite[Theorem 6]{Bar-NatanTop} proved that the IHX relation is a consequence of the STU relation for diagrams 
of $\diagrammesc$. 

Let us now prove the STU relation.
Let $\Gamma_J$, $\Gamma_J^{(1)}$ and $\Gamma_J^{(2)}$ be three Jacobi diagrams connected by the 
STU relation of Definition \ref{ASSTU}, as in Figure \ref{figureSTU}. 
We are going to prove that $w_{BCR}(\Gamma_J) = w_{BCR}(\Gamma_J^{(1)}) - w_{BCR}(\Gamma_J^{(2)})$. 
Let the vertices $v$, $w$, $t$ and $u$ and the edges $e$, $f$ and $h$ be as in Figure \ref{figureSTU}.
\begin{figure}[H]\centering
\gammatrois \ \ \gammaun \ \ \ \ \gammadeux
\caption{}
\label{figureSTU}
\end{figure}

We have a natural identification $E(\Gamma_J^{(1)})\cong E(\Gamma_J^{(2)}) \cong E(\Gamma_J)\setminus \{h\}$.
Fix a numbering $j_1$ of $\Gamma_J^{(1)}$, and let $j_2$ the associated numbering of $\Gamma_J^{(2)}$.
Fix $i_0\in\und{3k}\setminus j_1(E(\Gamma_J^{(1)}))$ and let $j$ denote the numbering of $\Gamma_J$ such that $j(h) = i_0$ and that induces $j_1$ on $\Gamma_J^{(1)}$.

For $i\in\{1,2\}$, split $\lettre(\Gamma_J^{(i)}, j_i)$ into $\lettre_1(\Gamma_J^{(i)}, j_i)$ and $\lettre_2(\Gamma_J^{(i)}, j_i)$, where 
\[\lettre_1(\Gamma_J^{(i)}, j_i)= \{(\Gamma,\sigma,\rho)\in\lettre(\Gamma_J^{(i)}, j_i) \mid
\text{There is exactly one internal edge between $v$ and $w$}\},\]
\[\lettre_2(\Gamma_J^{(i)}, j_i)=
\lettre(\Gamma_J^{(i)}, j_i)\setminus\lettre_1(\Gamma_J^{(i)}, j_i) .\]
For any $(\Gamma, \sigma, \rho)\in \lettre(\Gamma_J^{(1)}, j_1)$, 
let $\rho^*$ denote the ordering $\rho \circ \rho_{v,w}$, 
where $\rho_{v,w}$ is the transposition of $v$ and $w$. 
This induces a bijection $\big((\Gamma, \sigma, \rho)\in \lettre(\Gamma_J^{(1)}, j_1) \mapsto 
(\Gamma, \sigma, \rho^*)\in \lettre(\Gamma_J^{(2)}, j_2)
\big)$, which preserves the above decomposition 
$\lettre(\Gamma_J^{(i)}, j_i)=
\lettre_1(\Gamma_J^{(i)}, j_i)\sqcup\lettre_2(\Gamma_J^{(i)}, j_i)$.
Note that $\epsilon_3(\Gamma_J^{(1)},\Gamma, \rho) = \epsilon_3(\Gamma_J^{(2)},\Gamma,\rho^*)$ 
and that
 \[\epsilon_2(\Gamma,\rho^*) = \begin{cases} -\epsilon_2(\Gamma,\rho) & \text{if $(\Gamma,\sigma,\rho) \in \lettre_1(\Gamma_J^{(1)}, j_1)$,}\\
 \epsilon_2(\Gamma,\rho) & \text{if $(\Gamma,\sigma,\rho) \in \lettre_2(\Gamma_J^{(1)}, j_1)$.}\end{cases}\]
 This yields $w_{BCR}(\Gamma_J^{(1)}) - w_{BCR}(\Gamma_J^{(2)}) 
 = 2\sum\limits_{(\Gamma,\sigma,\rho)\in \lettre_1(\Gamma_J^{(1)}, j_1)} 
\frac{\epsilon(\Gamma)\epsilon_2(\Gamma,\rho)\epsilon_3(\Gamma_J^{(1)},\Gamma,\rho)}{2^{2k-\Card\left(E\left(\Gamma_J^{(1)}\right)\right)}}.$

Let $\lettre_1^a(\Gamma_J^{(1)}, j_1)$ denote 
the set of ordered and numbered BCR diagrams $(\Gamma,\sigma, \rho)\in \lettre_1(\Gamma_J^{(1)}, j_1)$ 
such that $v$ and $w$ are both trivalent in $\Gamma$, and set 
$\lettre_1^b(\Gamma_J^{(1)}, j_1)= \lettre_1(\Gamma_J^{(1)}, j_1)\setminus \lettre_1^{a}(\Gamma_J^{(1)}, j_1)$.
For any $(\Gamma,\sigma, \rho)\in \lettre_1^{a}(\Gamma_J^{(1)}, j_1)$, 
let $x$ and $y$ denote the univalent vertices respectively adjacent to $v$ and $w$ as in Figure \ref{gunaa},
and set $\rho^* = \rho\circ \rho_{x,y}\circ \rho_{v,w}$ and $\sigma^* = \sigma\circ \rho_{e,f}$.
Since there is only one internal edge from $v$ to $w$, 
there are internal edges $(v',v)$ and $(w, w')$ where $v'$ and $w'$ are neither $v$ nor $w$,
and $\epsilon_2(\Gamma,\rho^*) = - \epsilon_2(\Gamma,\rho)$.
The orientation of the possible external vertices did not change, so
$\epsilon_3(\Gamma_J^{(1)},\Gamma,\rho) = \epsilon_3(\Gamma_J^{(1)},\Gamma, \rho^*)$. 
Since $\big((\Gamma,\sigma,\rho)\in \lettre_1^a(\Gamma_J^{(1)}, j_1)
 \mapsto (\Gamma,\sigma^*, \rho^*) \in \lettre_1^a(\Gamma_J^{(1)}, j_1)\big)$
is a bijection, this yields 
\[\sum\limits_{(\Gamma,\sigma,\rho)\in \lettre_1^a(\Gamma_J^{(1)}, j_1)} 
\frac{\epsilon(\Gamma)\epsilon_2(\Gamma,\rho)\epsilon_3(\Gamma_J^{(1)},\Gamma,\rho)}{2^{2k-\Card(E(\Gamma_J^{(1)}))}}=0.\]

\begin{figure}[H]
\centering
\guna
\caption{Notations for a graph of $\lettre_1^a(\Gamma_J^{(1)}, j_1)$.}\label{gunaa}
\end{figure}

We now define a bijection \[\big((\Gamma,\sigma,\rho) \in \lettre^b_1(\Gamma_J^{(1)}, j_1) 
\mapsto (\Gamma^*, \sigma^*, \rho^*) \in \lettre(\Gamma_J, j)\big).\]
For any $(\Gamma, \sigma, \rho) \in \lettre_1^b(\Gamma_J^{(1)}, j_1)$, let $\Gamma^*$
denote the BCR diagram defined as follows: \begin{itemize}
\item if $v$ and $w$ are both bivalent, $\Gamma^*$ is obtained as in Figure \ref{figureBB},
\begin{figure}[H]
\centering
\gammaBBa \ \ \ \ \gammaBBastar \\
\gammaBBb \ \ \ \ \gammaBBbstar 
\caption{}\label{figureBB}
\end{figure}

\item if $v$ is trivalent and $w$ bivalent, $\Gamma^*$ is obtained as in Figure \ref{figureTB},
\begin{figure}[H]
\centering
\gammaTBa \ \ \ \ \gammaTBastar \\
\gammaTBb \ \ \ \ \gammaTBbstar 
\caption{}\label{figureTB}
\end{figure}
\item if $v$ is bivalent and $w$ trivalent, $\Gamma^*$ is obtained as in Figure \ref{figureBT},

\begin{figure}[H]
\centering
\gammaBTa \ \ \ \ \gammaBTastar \\
\gammaBTb \ \ \ \ \gammaBTbstar 
\caption{}\label{figureBT}
\end{figure}
\end{itemize}
The ordering $\rho$ induces a natural ordering $\rho^*$ of $\Gamma^*$. Note that $E_e(\Gamma^*)\setminus\{h\}\cong E_e(\Gamma)$. 
Let $\sigma^*$ denote the numbering of $\Gamma^*$ that coincide with $\sigma$ on $E_e(\Gamma^*)\setminus\{h\}$ and such that $\sigma^*(h) =i_0$,
so that $(\Gamma^*, \sigma^*, \rho^*) \in \lettre(\Gamma_J, j)$. 
Let us check that $\epsilon(\Gamma)\epsilon_2(\Gamma,\rho)\epsilon_3(\Gamma_J^{(1)},\Gamma, \rho) = 
\epsilon(\Gamma^*)\epsilon_2(\Gamma^*,\rho^*)\epsilon_3(\Gamma_J,\Gamma^*, \rho^*)$.

\begin{itemize}
\item If $v$ and $w$ are bivalent, since $\Gamma^*$ has one more external edge and one more trivalent vertex than $\Gamma$, $\epsilon(\Gamma)= \epsilon(\Gamma^*)$. 
\begin{itemize}
\item If there is an internal edge from $v$ to $w$ as in the first row of Figure \ref{figureBB}, 
$\epsilon_2(\Gamma^*,\rho^*) = \epsilon_2(\Gamma, \rho)$ since $\rho(w)-\rho(v)$ is positive. 
The orientation of the vertex $t$ is given by the cyclic order $(e,h,f)$ as in $\Gamma_J$, 
and $\epsilon_3(\Gamma_J,\Gamma^*, \rho^*) = \epsilon_3(\Gamma_J^{(1)}, \Gamma,\rho)$.
\item If there is an internal edge from $w$ to $v$ as in the second row of Figure \ref{figureBB}, 
$\epsilon_2(\Gamma^*, \rho^*) = -\epsilon_2(\Gamma, \rho)$ since $\rho(v)-\rho(w)$ is negative. 
The orientation of the vertex $t$ is given by the cyclic order $(f,h,e)$. 
It is the opposite of the orientation of $t$ in $\Gamma_J$, 
and $\epsilon_3(\Gamma_J,\Gamma^*, \rho^*) =- \epsilon_3(\Gamma_J^{(1)}, \Gamma,\rho)$.
\end{itemize}
\item If $v$ is trivalent and $w$ is bivalent, since $\Gamma^*$ has one more external edge than $\Gamma$ and as much trivalent vertices, $\epsilon(\Gamma)=- \epsilon(\Gamma^*)$.
\begin{itemize}
\item If there is an internal edge from $v$ to $w$ as in the first row of Figure \ref{figureTB}, 
$\epsilon_2(\Gamma^*,\rho^*) = \epsilon_2(\Gamma, \rho)$ since $\rho(w)-\rho(v)$ is positive. 
The orientation of the vertex $t$ is given by the cyclic order $(h,e,f)$.
It is the opposite of the orientation of $t$ in $\Gamma_J$, 
and $\epsilon_3(\Gamma_J,\Gamma^*, \rho^*) = -\epsilon_3(\Gamma_J^{(1)}, \Gamma,\rho)$.
\item If there is an internal edge from $w$ to $v$ as in the second row of Figure \ref{figureTB}, 
$\epsilon_2(\Gamma^*, \rho^*) = -\epsilon_2(\Gamma, \rho)$ since $\rho(v)-\rho(w)$ is negative. 
The orientation of the vertex $t$ is given by the cyclic order $(f,e,h)$. 
It is the orientation of $t$ in $\Gamma_J$, 
and $\epsilon_3(\Gamma_J, \Gamma^*, \rho^*) = \epsilon_3(\Gamma_J^{(1)}, \Gamma,\rho)$.
\end{itemize}
\item If $v$ is bivalent and $w$ is trivalent, 
the same argument yields $\epsilon(\Gamma^*) = -\epsilon(\Gamma)$, 
and we prove as above that 
$\epsilon_2(\Gamma^*,\rho^*)\epsilon_3(\Gamma_J,\Gamma^*,\rho^*) =
- \epsilon_2(\Gamma, \rho)\epsilon_3(\Gamma_J^{(1)}, \Gamma,\rho)$ in both cases of Figure \ref{figureBT}.

\end{itemize}

Since $\Gamma_J$ has one edge more than $\Gamma_J^{(1)}$, 
\begin{eqnarray*}
\sum\limits_{(\Gamma,\sigma,\rho)\in \lettre_1^b(\Gamma_J^{(1)}, j_1)} 
\frac{\epsilon(\Gamma)\epsilon_2(\Gamma,\rho)\epsilon_3(\Gamma_J^{(1)},\Gamma,\rho)}{2^{2k-1-\Card(E(\Gamma_J^{(1)}))}}
&=&
\sum\limits_{(\Gamma^*,\sigma^*,\rho^*)\in\lettre(\Gamma_J, j)}
\frac{\epsilon(\Gamma^*)\epsilon_2(\Gamma^*,\rho^*)\epsilon_3(\Gamma_J,\Gamma^*,\rho^*)}{2^{2k-\Card(E(\Gamma_J))}}\\
&=&w_{BCR}(\Gamma_J).\end{eqnarray*}
This yields $w_{BCR}(\Gamma_J^{(1)}) -w_{BCR}(\Gamma_J^{(2)}) = w_{BCR}(\Gamma_J)$ and concludes the proof of Lemma \ref{STU}.\qedhere
\end{proof}
\subsection{Restriction to connected diagrams}

\begin{lm}\label{conn}
If $\Gamma_J$ is a non-trivial product of diagrams, then $w_{BCR}(\Gamma_J) = 0$.
\end{lm}
\begin{proof}
Let $\Gamma_J$ be a non-trivial product, and let $\Gamma_J^{(1)}$ and $\Gamma_J^{(2)}$
be two Jacobi diagrams such that $\Gamma_J$ is the product of $\Gamma_J^{(1)}$ and $\Gamma_J^{(2)}$. Let $j$ be a numbering of $\Gamma_J$.
We are going to define an involution $\left((\Gamma,\sigma, \rho)\in \lettre(\Gamma_J, j) \mapsto
(\Gamma^*, \sigma^*, \rho^*) \in \lettre(\Gamma_J, j)\right)$.

For any $(\Gamma, \sigma,\rho)\in \lettre(\Gamma_J, j)$, 
let $\sommets = V_1\sqcup V_2$ be the partition of $\sommets$ such that 
the vertices of $V_i$ correspond to the vertices of $\Gamma_J^{(i)}$ in $\Gamma_J$.
Let $\Gamma'$ denote the graph obtained from $\Gamma$ by keeping only the vertices of $V_2$ and the edges
between such vertices, and
let $\Gamma_c$ denote the connected component of $\Gamma'$ that contains the external edge with minimal
$\sigma$.
Let $V_3$ denote the set of vertices of $\Gamma_c$.
By construction of BCR diagrams, there is exactly one internal edge $e_1=(v,w)$ from $V_1$ to an internal vertex $w$ of $V_3$.
Note that any element of $\rho(V_1\cap V_i(\Gamma))$ is before any element of $\rho(V_3\cap V_i(\Gamma))$

Let $\Gamma^*$ denote the BCR diagram defined as follows. \begin{itemize}

\item If $w$ is bivalent in $\Gamma$, 
and if we have an external edge $e=(w, w_2)$ from $w$ to another bivalent vertex $w_2$, then $w_2$ is in $V_3$ since $\Gamma_c$ is connected. Set $w_1=w$. The graph
$\Gamma^*$ is obtained from $\Gamma$ after replacing 
the internal edge $e_1=(v,w_1)$ with an internal edge from $v$ to $w_2$ as in Figure \ref{figgncun}, and $\rho^*$ 
and $\sigma^*$ are naturally deduced from $\rho$ and $\sigma$. 
Since $w_1$ and $w_2$ are in $V_3$ and $v$ in $V_1$, 
$\rho(w_2)-\rho(v)$ and $\rho(w_1)-\rho(v)$ have the same sign, and
$\epsilon_2(\Gamma^*,\rho^*) = \epsilon_2(\Gamma,\rho)$. Furthermore, we have $\epsilon(\Gamma^*) = -\epsilon(\Gamma)$ since $\Gamma^*$ has the same external edges as $\Gamma$ but one more (internal) trivalent vertex. 
Since nothing changed around the external trivalent vertices, $\epsilon_3(\Gamma_J,\Gamma^*,\rho^*) = \epsilon_3(\Gamma_J, \Gamma,\rho)$.
\item If $w$ is trivalent in $\Gamma$, set $w_2=w$.
There is a leg $e=(w_1,w_2)$ from a univalent vertex $w_1$ to $w_2$. 
Since $\Gamma_c$ is connected, $w_1$ is in $V_3$.
In this case, $\Gamma^*$
is the graph obtained from $\Gamma$ after replacing the internal edge $e_1=(v, w_2)$ with $(v,w_1)$
 as in Figure \ref{figgncun},
and $\rho^*$ and $\sigma^*$ are naturally determined by $\rho$ and $\sigma$. 
As above, we have $(\epsilon(\Gamma^*),\epsilon_2(\Gamma^*,\rho^*), \epsilon_3(\Gamma_J,\Gamma^*, \rho^*))
= (-\epsilon(\Gamma),\epsilon_2(\Gamma,\rho), \epsilon_3(\Gamma_J,\Gamma, \rho))$.
\begin{figure}[H]
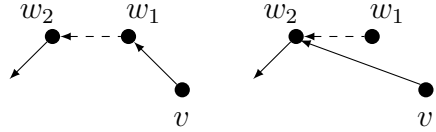

\centering
\captionsetup{justification=centering}
\gnccyc \ \ \ \gncleg
\centering
\caption{Notations for the first two cases.\\
In the first case, $\Gamma$ is on the left and $\Gamma^*$ on the right. \\
In the second case, this is the other way.}\label{figgncun}
\end{figure}

\item Otherwise, $\Gamma_c$ is as in Figure \ref{Fgnc} and $w$ is connected to an external trivalent vertex $t$, where a leg from a univalent vertex $x$ arrives. 
In this case, $\Gamma^*= \Gamma$ and $\rho^* = \rho \circ \rho_{x,w}$, where $\rho_{x,w}$ is the transposition of $x$ and $w$. 
We have $\rho^*(w) - \rho^*(v)= \rho(x)-\rho(v)$.
This expression has the same (positive) sign as 
$\rho(w)-\rho(v)$ since $x$ and $w$ are in $V_3$ and $v$ is in $V_1$. 
Therefore, $\epsilon_2(\Gamma,\rho^*) = \epsilon_2(\Gamma,\rho)$. 
Since we only changed the order of two internal vertices adjacent to the trivalent vertex $t$, 
we have $\epsilon_3(\Gamma_J,\Gamma, \rho^*) = -\epsilon_3(\Gamma_J,\Gamma,\rho)$.
\end{itemize}
\begin{figure}[H]
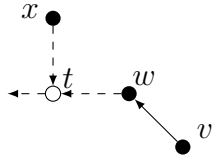

\centering
\gnc
\caption{Notations for the third case.}\label{Fgnc}
\end{figure}

This yields an involution $\left((\Gamma,\sigma, \rho)\in \lettre(\Gamma_J, j) 
\mapsto (\Gamma^*, \sigma^*, \rho^*) \in \lettre(\Gamma_J, j)\right)$ as announced
and we have $\epsilon(\Gamma^*)\epsilon_2(\Gamma^*,\rho^*)\epsilon_3(\Gamma_J,\Gamma^*, \rho^*)=
- \epsilon(\Gamma) \epsilon_2(\Gamma,\rho)\epsilon_3(\Gamma_J,\Gamma,\rho)$ for any $(\Gamma,\sigma, \rho)
\in \lettre(\Gamma_J, j)$. This concludes the proof of the lemma.\qedhere
\end{proof}

Lemmas \ref{calculgk}, \ref{STU} and \ref{conn} and Lemma \ref{wc} conclude the proof of Proposition \ref{th2} so that Theorem \ref{th1} is proved.

	\bibliographystyle{alpha}
	\bibliography{article3}
\end{document}